\documentclass[12pt]{amsart}
\usepackage{times}
\usepackage{anysize}
\marginsize{2.7cm}{2.44cm}{2.0cm}{3.0cm}
\usepackage[small,it]{caption}
 \usepackage[latin1]{inputenc}
 \usepackage[dvips]{graphicx}
 \usepackage{wrapfig}
 \usepackage{amsmath}
 \usepackage{amsthm}
 \usepackage{amsfonts}
 \usepackage{amssymb}
 \usepackage{layout
} \usepackage{verbatim}
 \usepackage{alltt}
\usepackage{yfonts}

\usepackage[all]{xy}
\usepackage{setspace}

\newcommand{\ord}{\textup{ord}}

\newcommand{\LL}{\Lambda}
\newcommand{\LLac}{\Lambda^{\textup{ac}}}
\newcommand{\TT}{\mathbb{T}}
\newcommand{\QQ}{\mathbb{Q}}
\newcommand{\FF}{\mathcal{F}}
\newcommand{\VV}{\mathbb{V}}

\newcommand{\lra}{\longrightarrow}
\newcommand{\ZZ}{\mathbb{Z}}
\newcommand{\PP}{\mathcal{P}}

\newcommand{\NN}{\mathcal{N}}
\newcommand{\ra}{\rightarrow}

\newcommand{\be}{\begin{equation}}
\newcommand{\ee}{\end{equation}}

\newcommand{\al}{\mathcal{L}}

\newcommand{\oo}{\mathcal{O}}
\newcommand{\RR}{\mathcal{R}}

\newcommand{\mm}{\hbox{\frakfamily m}}

\numberwithin{equation}{section}
\newtheorem{thm}{Theorem}[section]
\newtheorem{lemma}[thm]{Lemma}
\newenvironment{define}{\par\medskip\noindent\refstepcounter{thm}
\bgroup{\hspace*{-0.15 cm}\bf{Definition}
\thethm.}\bgroup}{\egroup \egroup\par\medskip}
\newtheorem{prop}[thm]{Proposition}
\newtheorem{cor}[thm]{Corollary}
\newenvironment{rem}{\par\medskip\noindent\refstepcounter{thm}
\bgroup{\hspace*{-0.15 cm}\bf{Remark} \thethm.}\bgroup}{\egroup
\egroup\par\medskip} \DeclareMathOperator{\id}{id}\parskip 2pt

\newenvironment{hypo}{\par\medskip\noindent\refstepcounter{thm}
\bgroup{\hspace*{-0.15 cm}\bf{Hypothesis}
\thethm.}\bgroup}{\egroup \egroup\par\medskip}

\newenvironment{assume}{\par\medskip\noindent\refstepcounter{thm}
\bgroup{\hspace*{-0.15 cm}\bf{Assumption}
\thethm.}\bgroup}{\egroup \egroup\par\medskip}
\newcounter{Athm}[section]

\renewcommand{\theAthm} {\arabic{Athm}}

\newenvironment{Athm}{\par\medskip\noindent\refstepcounter{Athm}
\bgroup{\hspace*{-0.15 cm}\bf{Theorem}
A.\theAthm.}\bgroup\it}{\egroup \egroup\par\medskip}

\begin{document}
\title{{B}\lowercase{ig} {H}\lowercase{eegner point} {K}\lowercase{olyvagin system for a family of modular forms}}

\author{K\^az\i m B\"uy\"ukboduk}

\address{K\^az\i m B\"uy\"ukboduk \hfill\break\indent Ko\c{c} University, Mathematics  \hfill\break\indent Rumeli Feneri Yolu \hfill\break\indent 34450 Sariyer, Istanbul
\hfill\break\indent Turkey
\hfill\break\indent 
\vskip0.15cm}
%\address{PUC Facultad de Matem\'aticas,    \hfill\break\indent 
%Campus San Joaqu\'{i}n
%\hfill\break\indent Avenida Vicu\~{n}a Mackenna 4860 \hfill\break\indent Santiago
%\hfill\break\indent Chile}

%\curraddr{Ko\c{c} University,  Mathematics \hfill\break\indent Rumeli Feneri Yolu \hfill\break\indent 34450 Sar\i yer/\.Istanbul
%\hfill\break\indent Turkey}
%\address{\textit{Current Address:} \hfill\break\indent
%Ih\'es, Le Bois-Marie, 35,  \hfill\break\indent Route de Chartres \hfill\break\indent F-91440 Bures-sur-Yvette
%\hfill\break\indent France}
%\keywords{Stark conjectures, Euler systems, Kolyvagin systems.}

\begin{abstract}
The principal goal of this paper is to develop Kolyvagin's descent
to apply with the big Heegner point Euler system constructed by
Howard\,%~\cite{howard}
for the big Galois representation $\TT$ attached to a Hida family $\mathbb{F}$ of
elliptic modular forms. In order to achieve this,
we interpolate and control the Tamagawa factors attached to each
member of the family $\mathbb{F}$ at bad primes, which should be of
independent interest. Using this, we then work out the Kolyvagin
descent on the big Heegner point Euler system so as to obtain a big Kolyvagin system that interpolates the collection of Kolyvagin systems obtained by Fouquet for each member of the family individually. This construction has standard applications to Iwasawa theory, which we record at the end.
\end{abstract}

\maketitle
%\tableofcontents
\section{Introduction}
\label{sec:intro} The main goal of this article is to develop a Kolyvagin descent
procedure for the big Heegner point Euler system constructed by
Howard in \cite{howard}, associated to a Hida family of ordinary
modular forms. This we achieve under the hypothesis that the family
passes through a single (twisted) eigenform\footnote{In the sense of \cite[\S12.7.10]{nek}. The twisted eigenform we have in mind is denoted by $g_{\mathcal{P}}$ in loc.cit.} whose Tamagawa factors at bad
primes are coprime to $p$. Through this construction, we obtain a
big Kolyvagin system for the big Galois representation, with
standard applications. Before stating our results, we start with
setting up the notation.

 Let $N$ be a positive integer and $p \nmid 6N$ a
prime. Define
$$\omega: \Delta=(\ZZ/p\ZZ)^\times \lra \pmb{\mu}_{p-1}$$
to be the Teichm\"uller character, which we view both as a $p$-adic
and complex character by fixing embeddings $\overline{\QQ}
\hookrightarrow \overline{\QQ}_p$, $\overline{\QQ} \hookrightarrow
\mathbb{C}$, as well as a Dirichlet character modulo $Np$. Let
$$f=\sum_{n=1}^{\infty} a_nq^n \in S_k(\Gamma_0(Np),{\omega}^j)$$
be a normalized cusp form of weight $k \geq 2$, which is an
eigenform for the Hecke operators $T_\ell$ for $\ell\nmid Np$ and
$U_\ell$ for $\ell \mid Np$. Let $E/\QQ_p$ be a finite extension
that contains $a_n$ for all $n$ and let $\oo=\oo_E$ be its ring of integers and $\pi=\pi_E$ a fixed uniformizer. We assume further that $f$ is an eigenform that is
$p$-ordinary and $p$-stabilized, and the conductor of $f$ is
divisible by $N$. This amounts to saying that $a_p \in \oo^\times$
and the system of Hecke-eigenvalues $\{a_\ell \mid \ell \nmid Np\}$
associated to $f$ agrees with that of a newform of level $N$ or
$Np$. Let $G_\QQ=\textup{Gal}(\overline{\QQ}/{\QQ})$ and let
$\rho_f: G_\QQ \ra \textup{GL}_2(E)$ be the Galois representation
attached to $f$ by Deligne~\cite{deligne69}.  Throughout this paper,
we assume the following holds:
\begin{hypo}
\label{hypo1} The semi-simple residual representation
$\overline{\rho}_f$ associated to $\rho_f$ is absolutely
irreducible and is $p$-distinguished.
\end{hypo}
Let $\Gamma=1+p\ZZ_p$. Identify $\Delta$ with $\pmb{\mu}_{p-1}$ via $\omega$ so that we have
$$\ZZ_p^\times\cong \Delta\times\Gamma.$$
Set $\LL=\oo[[\Gamma]]$. Let $\frak{h}^{\textup{ord}}$ be Hida's
universal ordinary algebra parametrizing the Hida family passing through
$f$, which is finite flat over $\LL$ by~\cite[Theorem
1.1]{hidainv86}. We will recall some basic properties of
$\frak{h}^{\textup{ord}}$, for details the reader may
consult~\cite{hidainv86, hidaens86} and \cite{emertonpollackweston}
for an excellent quick survey. The eigenform $f$ fixed as above
corresponds to an \emph{arithmetic specialization} (in the sense of
Definition~\ref{def:arithmeticspecialization} below)
$$\frak{s}_f: \frak{h}^{\textup{ord}}\lra \oo$$
%which is characterized by $\frak{s}_f(T_\ell)=a_\ell$ for $\ell \nmid Np$ and $\frak{s}_f(U_\ell)=a_\ell$ for $\ell \mid Np$.
Decompose $\frak{h}^{\textup{ord}}$ into a direct sum of its
completions at maximal ideals and let
$\frak{h}^{\textup{ord}}_{\mm}$ be the (unique) summand through
which $\frak{s}_f$ factors. The localization of
$\frak{h}^{\textup{ord}}$ at $\ker(\frak{s}_f)$ is a discrete
valuation ring~\cite[\S12.7.5]{nek}, and hence there is a unique
minimal prime $\frak{a} \subset \frak{h}^{\textup{ord}}_{\mm}$ such
that $\frak{s}_f$ factors through the integral domain \be
\label{eqn:defR} \mathcal{R}=\frak{h}^{\textup{ord}}_{\mm}/\frak{a}.
\ee

The $\LL$-algebra $\mathcal{R}$ is called the branch of the Hida
family on which $f$ lives, by duality it corresponds to a family $\mathbb{F}$ of ordinary modular forms. Hida~\cite{hidaens86} gives a
construction of a  big $G_{\QQ}$-representation $\mathbf{T}$ with
coefficients in $\mathcal{R}$, the exact definition of $\mathbf{T}$ is
recalled below. Thanks to Hypothesis~\ref{hypo1}, $\mathbf{T}$
is a free $\RR$-module of rank two. The
$G_{\QQ}$-representation $\mathbf{T}$ is unramified outside $Np$.
Let $\TT$ be the critical twist of $\mathbf{T}$, as defined by
Howard~\cite[Definition 2.1.3]{howard}. Then the
$G_\QQ$-representation is self-dual in the sense that there is a
perfect $\mathcal{R}$-bilinear pairing
$$\TT\times\TT \lra \mathcal{R}(1).$$
Fix a quadratic imaginary number field $K$ and let $\oo_K$ b its
ring of integers.  Assume until the end that the following holds:
\begin{hypo}
\label{hypo2}
\begin{itemize}
\item[(i)]There is an ideal $\frak{N}$ of $\oo_K$ such that $\oo_K/\frak{N}
\cong \ZZ/N\ZZ$.
\item[(ii)] The class number of $K$ is prime to $p$.
\end{itemize}
\end{hypo}
%Self-note: Implicitly we assume that N is square-free (Nekovar's assumption H_v requires that) and by (i) that all primes $\ell$ that divide N split in K. Thus, for K_v that appear below (v | N), we may put \QQ_v.
Let $H_{c}$ be the ring class field of $K$ of conductor ${c}$ and for $c$ prime to $p$ (resp., for $\alpha \in \ZZ^+$), let $K(c)$ (resp., $K_\alpha$) be the maximal $p$-extension in $H_c/K$ (resp., in $H_{p^{\alpha+1}}/K$). Set $K_\alpha(c)$ to be the composite field of $K_\alpha$ and $K(c)$, $K_\infty=\cup_\alpha K_\alpha$, $\Gamma^{\textup{ac}}=\textup{Gal}(K_\infty/K)$ and $\LL^{\textup{ac}}=\ZZ_p[[\Gamma^{\textup{ac}}]]$.
In~\cite[\S2.2]{howard}, Howard constructs a family of cohomology
classes $\frak{X}_{c} \in \tilde{H}_f^1(H_{c}, \TT),$
where $\tilde{H}_f^1(H_{c}, \TT)$ is Nekov\'a\v{r}'s \cite[\S6]{nek}
extended Selmer group. Howard also checks in Proposition 2.3.1 of loc.cit. that these
classes satisfy the Euler system relation. For $c$ prime to $p$, set 
$$\frak{z}_{c,\alpha}=\textup{cor}_{H_{cp^{\alpha+1}/K_\alpha(c)}} \left(U_p^{-\alpha} \frak{X}_{cp^{\alpha+1}}\right) \in \tilde{H}_f^1(K_\alpha(c),\TT).$$
This definition makes sense thanks to \cite[Proposition 2.3.1]{howard}. The collection $\{\frak{z}_{c,\alpha}\}_{_{c,\alpha}}$ is called the
\emph{big Heegner point Euler system.} To ease notation, write
$$\frak{z}_\alpha=\frak{z}_{1,\alpha}\in \tilde{H}_f^1(K_\alpha,\TT).$$ 
The collection $\{\frak{z_\alpha}\}$ is norm-compatible as $\alpha$ varies and we may therefore set $$\frak{z_\infty}=\{\frak{z}_\alpha\} \in \varprojlim_\alpha \tilde{H}_f^1(K_\alpha,\TT)=: \tilde{H}_{f,\textup{Iw}}^1(K_\infty,\TT).$$

The first (\textbf{H.stz})  of the following hypotheses may be thought of as an assumption to rule out the existence of exceptional zeros (in the sense of Greenberg~\cite{gr2}) at characters of $\Gamma^{\textup{ac}}$ of finite order. The second (\textbf{H.Tam}) has to do with Tamagawa factors.\\

(\textbf{H.stz}) For every $v|p$, $H^0(K_v,\textup{F}_v^-(\overline{T}))=0$.
\\\\Here $\overline{T}=\TT/\mm_{\RR}$ is the residual representation and $\textup{F}_v^-(\overline{T})$ is defined as in~\S\ref{sec:KSdescent}. See Remark~\ref{rem:hstzharsh?} for the content of this hypothesis. \\

(\textbf{H.Tam})
\begin{itemize}
\item [(i)] $\displaystyle{p\nmid \prod_{\ell\mid N} (\ell^2-1)}$,
\item [(ii)] there is a specialization $T$ of the twisted Hida family $\TT$ for which
$p\nmid c_\ell(T)$.\\
\end{itemize}

Starting from an Euler system for a Galois representation $M$ with
coefficients in a discrete valuation ring, Mazur and
Rubin~\cite{mr02} devise\footnote{Attentive reader will notice that
Mazur and Rubin never treat Heegner points. It was
Howard~\cite{howard-heegner1} who was the first to study the Heegner
points from the perspective offered by the work of Mazur and Rubin.}
a machinery which yields a \emph{Kolyvagin system} for $M$. This is
what we carry out for the big Galois representation $\TT$ which has
coefficients over a dimension-2 Gorenstein ring $\mathcal{R}$ and
prove the following:

\begin{Athm}\textup{[See Theorem~\ref{thm:KS}]}
\label{athm:KS} Suppose  the assumptions {\upshape{\textbf{H.Tam}}} and {\upshape\textbf{H.stz}} hold true. There is a Kolyvagin system 
$$\{\kappa_n\}=\pmb{\kappa} \in
\overline{\textbf{KS}}(\TT\otimes\LL^{\textup{ac}},\FF_{\textup{Gr}})$$
(where the $\mathcal{R}$-module
$\overline{\textbf{KS}}(\TT\otimes\LL^{\textup{ac}},\FF_{\textup{Gr}})$ is described in Definition
\ref{def:bigKS} below) such that
$$\kappa_1=\frak{z}_\infty \in \tilde{H}^1_{f,\textup{Iw}}(K_\infty,\TT).$$
\end{Athm}

Assuming ({\upshape{\textbf{H.Tam})(ii)}} alone, Howard in \cite[Theorem 2.4.5]{howard} proves that the classes $\{\frak{z}_{c,\alpha}\}$ lie in the Greenberg Selmer group. However, to carry out the descent argument (as we do in \S\ref{sec:KSdescent}) in order to deduce Theorem A.1, one needs the finer analysis of local cohomology groups that we carry out in \S\ref{sec:tamcontrol}. As a by-product to our analysis we control, among other things, the variation of Tamagawa factors in the ``Hida family"  $\TT$, much in the spirit of \cite{emertonpollackweston}. To that end, we show for a prime $\ell$ that divides the tame conductor $N$, how to interpolate the Tamagawa factors $\{c_\ell(T_S)\}_{\frak{s}}$ into an element $\pmb{\tau}$ of $\mathcal{R}$ (which we call the \emph{Tamagawa element}). Here $\frak{s}$ runs
through specializations $$\frak{s}: \mathcal{R} \lra S,\,\,\,\,\,\,\,\, T_S=\TT\otimes_{\mathcal{R}}S$$
into discrete valuation rings $S$ and the Tamagawa factor $c_\ell(T_S)$ is defined following Fontaine and Perrin-Riou~\cite{FPR91}.  More precisely, we prove the following in \S\ref{sec:tamcontrol}:

\begin{Athm}
\label{thm:tamvariation}
\begin{itemize}
\item[(i)] There exists an element $\pmb{\tau} \in \RR$ such that $[S:\frak{s}(\pmb{\tau})S]=c_\ell(T_S)$.
\item[(ii)] Assume that the hypothesis {\upshape\textbf{H.Tam}} holds true. Then for any specialization $T_S$ of $\TT$, the Tamagawa factor $c_\ell(T_S)$ is coprime to $p$ as well.
\end{itemize}
\end{Athm}

See \S\ref{sec:tamcontrol} below for a precise definition of the element $\pmb{\tau}$ of $\RR$, which might be of independent interest. We remark that Theorem A.2(ii) can be obtained without going through the construction of the element $\pmb{\tau}$, c.f., \cite[Propositions 2.2.4 and 2.2.5]{emertonpollackweston}, \cite[Theorem 3.3]{ochiai06}, \cite[Lemma 2.14]{fochiai}. However, in order to descend to a Kolyvagin system, one needs the finer analysis in \S\ref{sec:tamcontrol}.

Once we obtain a Kolyvagin system as in Theorem A.1, a standard argument
(c.f., \cite{o, fouquetRIMS}) gives
bounds on the appropriate extended Selmer group. Suppose the ring $R_\infty:=\RR\otimes_{\ZZ_p}\LL^{\textup{ac}}$ is a regular ring and $M$ is a torsion $R_\infty$-module. We define the characteristic ideal of $M$ to be
$$\textup{char}(M)=\prod_{\frak{p}}\frak{p}^{\textup{length}(M_{\frak{p}})}$$
where the product runs through height-1 primes of $R_\infty$.

\begin{Athm}\textup{[See Theorem \ref{thm:mainappl}]}
\label{athm:KSapp} Suppose  the assumptions of Theorem~A.\ref{athm:KS} hold true and assume that the ring $R_\infty$ is regular. Then 
$$\textup{char}\left(\tilde{H}^2_{f,\textup{Iw}}(K_\infty, \TT)_{\textup{tors}}\right)\,\mid \, \textup{char}\left(\tilde{H}^1_{f,\textup{Iw}}(K_\infty, \TT)/R_\infty\frak{z}_\infty\right)^2.$$
\end{Athm}

We note that Fouquet has also devised a machinery to
make use of the big Heegner point Euler system and has obtained Theorem A.\ref{athm:KSapp}~\cite[Theorem 2.9]{fouquetRIMS}. In his statement, however, the class $\frak{z}_\infty$ here has to be replaced by a multiple by an element $\alpha \in \RR$, accounting for uncontrolled Tamagawa factors. We are able to control the uncontrolled ``extra" factor $\alpha$ thanks to Theorem~A.\ref{athm:KS} and thus
improve on Fouquet's result towards Howard's two-variable main
conjecture.  In fact, we expect that the big Kolyvagin system $\pmb{\kappa}$ we construct in Theorem A.\ref{athm:KS} (under the hypotheses \textbf{H.stz} and  \textbf{H.Tam}) is \emph{primitive} in a certain sense and that the divisibility in Theorem A.\ref{athm:KS} is sharp, which is not the case in~\cite{fouquetRIMS}. All this is a consequence of the fact that Fouquet's approach is all
together different than ours: Fouquet constructs Kolyvagin systems
for each individual specialization, whereas in this paper we prove
the existence of a big Kolyvagin system that essentially
interpolates each of these individual Kolyvagin systems. 

We finally remark that the hypothesis \textbf{H.stz} seems to be also necessary for the descent arguments in \cite{fouquetRIMS} to go through, see Remark~\ref{rem:hstzharsh?} below.

When the base field is a general totally real field, there also
exists a Hida family of Hilbert modular forms and the relevant
properties of the associated big Galois representation is
established in~\cite{hidaannals88, hida89algana, SW99, SW01}. A
construction of an Euler system of big Heegner points in this
setting is due to Fouquet~\cite{fouquetthesis}, generalizing the
work of Howard. The formalism of the current article applies to this more general setting as well, and that is one
reason why the author chose to stick to the case of elliptic modular
forms so as to keep the notation simple and various technical
constructions tractable.

Our construction of a big Heegner point Kolyvagin system (Theorem~A.\ref{athm:KS}) goes hand in hand with the main result of the forthcoming article~\cite{kbbdeform}, where one may better observe the benefits of deforming Kolyvagin systems directly (as opposed to deforming first the Euler system and then specializing to Kolyvagin systems for each individual member, as done so in \cite{fouquetRIMS}.) In \cite{kbbdeform}, we prove (generalizing the main theorem of \cite{kbb}) that the $\RR$-module of Kolyvagin systems for $\mathcal{T}$ is free of rank one for a very general class of Galois representations $\mathcal{T}$ with coefficients 
\begin{itemize}
\item either a 2-dimensional complete Gorenstein  ring $\RR$, 
\item or a regular complete Noetherian local ring $\RR$.
\end{itemize}
In particular, this result may be used to interpolate Kolyvagin systems obtained from Kato's Euler systems for elliptic modular forms to the universal deformation ring, under mild hypotheses. Such a construction seems intractable for the time being in the level of Euler systems. Furthermore, one may hope to extend the arguments of \cite{kbbdeform} in order to interpolate Kato's Kolyvagin systems for each individual member of a finite slope (not necessarily $p$-ordinary) Coleman family, and incorporate this construction within Pottharst's~\cite{pottharstSelmer} non-ordinary Iwasawa theory.

\subsection{Notation and Hypotheses}
\label{subsec:notationhypo} For any field $F$, fix a separable
closure $\overline{F}$ of $F$ and write $G_F$ for the absolute
Galois group $\textup{Gal}(\overline{F}/F)$. For a continuous
$G_F$-representation $M$, we will denote by $H^i(F,M)=H^i(G_F,M)$
the cohomology group calculated with continuous cochains.

For an algebraic number field $L$ and a non-archimedean place $w$ of $L$, we write $G_w$ instead of $G_{L_w}$. We also denote any fixed decomposition subgroup (resp., inertia subgroup) at $w$ of $G_L$ by $D_w$ (resp., $I_w$).

For an abelian group $A$, let
$\hat{A}=\textup{Hom}(\textup{Hom}(A,\QQ_p/\ZZ_p),\QQ_p/\ZZ_p)$ be
its $p$-adic completion. If further $G_F$ acts on $A$, then for a
character $\chi$ of $G_F$ with values in $\oo$ (where $\oo$ is the
ring of integers of a finite extension of $\QQ_p$), let
$$A^\chi:=\{a \in \hat{A}\otimes_{\ZZ_p}\oo: g\cdot a=\chi(g)a \hbox{ for all } g\in G_F\}.$$
\section{A family of Galois representations}
\label{sec:setupGalRep}
In this section, we give the definition of the big Galois representation $\TT$ attached to a Hida family of elliptic modular forms $\mathbb{F}$. We mostly follow~\cite[\S 2]{howard} as we will heavily rely on his constructions later in this paper, and we use in this section the terminology set in loc. cit.  sometimes without giving the definition here.

Let $\frak{h}^{\textup{ord}}$ be Hida's big ordinary Hecke algebra of tame level $N$ and let $\LL= \oo[[\Gamma]]$ be the Iwasawa algebra as in the introduction. Then $\frak{h}_{\textup{ord}}$ can made into a $\LL$-algebra via the diamond action and it is finite flat over $\LL$ by~\cite[Theorem 1.1]{hidainv86}.
\begin{define}
\label{def:arithmeticspecialization}
\begin{itemize}
\item[(i)] Let $\iota$ be the natural inclusion $\ZZ_p^\times \ra \oo[[\ZZ_p^\times]]^\times$. The restriction of $\iota$ to $\Gamma$ will also be denoted by $\iota$.
\item[(ii)] If $A$ is any finitely generated commutative $\LL$-algebra then the $\oo_E$-algebra map $A \stackrel{s}{\ra} \overline{\QQ}_p$ is called an \emph{arithmetic specialization} if the composition
$$\Gamma \stackrel{\iota}{\lra} A^\times \stackrel{s}{\lra}  \overline{\QQ}_p^\times$$
has the form $\gamma \mapsto \psi(\gamma)\gamma^{r-2}$ for some integer $r\geq 2$ and some character $\psi$ of $\Gamma$ of finite order.
\item[(iii)] The kernel of an arithmetic specialization is called an \emph{arithmetic prime} of $A$. If $\wp$ is an arithmetic prime then the residue field $E_\wp:=A_\wp/\wp A_\wp$ is a finite extension of $E$. The composition
$$\Gamma \lra A^\times \lra E_\wp^\times$$
has the form $\gamma \mapsto \psi_\wp(\gamma)\gamma^{r-2}$ for a character $\psi_\wp:\Gamma \ra E_\wp^\times$. The character $\psi_\wp$ is called the wild character of $\wp$ and the integer $r$ is called the \emph{weight} of $\wp$.
\end{itemize}
\end{define}
Let $f \in S_k(\Gamma_0(Np),\omega^j)$ be a cusp form as in the introduction. As explained in Introduction, $f$ corresponds to an arithmetic specialization $\frak{s}_f$ which factors through
$$\mathcal{R}=\frak{h}^{\textup{ord}}_{\mm}/\frak{a}$$
for a unique maximal ideal $\mm \subset \frak{h}^{\textup{ord}}$ and a uniquely determined minimal prime ideal $\frak{a} \subset \frak{h}^{\textup{ord}}_{\mm}$.

For $s \in \ZZ^+$, set $\Phi_s=\Gamma_0(N)\cap \Gamma_1(p^s) \subset \textup{SL}_2(\ZZ)$ and let $Y_s$ denote the affine modular curve $Y_s$ classifying elliptic curves with $\Phi_s$-level structure. More precisely, $Y_s$ classifies triples $(E,C,\pi)$ consisting of an elliptic curve $E$, a cyclic subgroup $C$ of $E$ of order $N$ and a point $\pi$ on $E$ of exact order $p^s$. Let $X_s$ be its compactification and let $J_s$ be the Jacobian of $X_s$. Then there is a degeneracy map
$$\alpha: X_{s+1}\lra X_s$$
given by
$$(E,C,\pi) \mapsto (E,C,p\cdot\pi),$$
which induces a map
$$\alpha_*:J_{s+1}\lra J_s.$$
Let $T_p(J_s)$ be the $p$-adic Tate module. Via the Albenese action, Hecke operators acts on each  $T_p(J_s)$. Let $e^{\textup{ord}}=\lim U_p^{n!}$ be Hida's ordinary projector, set $T_p^{\textup{ord}}(J_s)=e^{\textup{ord}}(T_p(J_s))$ and
$$\mathbf{T}=(\varprojlim T_p^{\textup{ord}}(J_s))\otimes_{\frak{h}^{\textup{ord}}} \mathcal{R}.$$
Here the inverse limit is with respect to $\alpha_*$. There is a natural $\mathcal{R}$-linear $G_\QQ$-action on $\mathbf{T}$. As indicated in Introduction, $\mathbf{T}$ is a free $\mathcal{R}$-module of rank two and is unramified outside $Np$. Let $\TT$ be the self-dual twist of $\mathbf{T}$, defined as in~\cite[Definition 2.1.3]{howard} (and denoted by $\mathbf{T}^\dagger$ in loc. cit.).

For any arithmetic prime $\wp\subset \mathcal{R}$, Hida Theory associates an ordinary modular form $f_\wp$ with coefficients in $E_\wp$. The $G_\QQ$-representation $V_\wp=\TT\otimes_{\mathcal{R}} E_{\wp}$ is then a self-dual twist of the $p$-adic Galois representation attached to the form $f_\wp$ by Deligne.
\section{Controlling the Tamagawa factors attached to families}
\label{sec:tamcontrol}
 Let $\RR$ and $\TT$ be as above and let $\frak{m}$ denote the maximal ideal of $\RR$. We write
$\al=\textup{Frac}(\RR)$, the field of fractions of $\RR$ and
$\VV=\TT\otimes_{\RR}\al$. Throughout this section $v \nmid p\infty$  will
denote a place of $K$ which divides the tame conductor $N$ of
the Hida family $\RR$.

For technical reasons we also impose the following:
\begin{assume}
\label{assume:p} $\displaystyle{p \nmid \prod_{v|N} (\mathbf{N}v^2-1).}$
\end{assume}

%With more work, it should be possible to drop this condition on $p$.
 \subsection{Local Galois representation at the primes dividing the tame conductor}
 \label{subsec:localgaloistame}
 Nekov\'a\v{r} in \cite[Proposition 12.7.14.1]{nek} describes the $\al[G_v]$-module $\VV$ under the following assumption:

 ($\mathbf{\textup{H}_v}$) There is a twisted cusp form $g_{\wp}$ through which the (twisted) Hida family passes through such that $\pi(g_\wp)_v=\textup{St}(\mu)$ (with $\mu^2=1$ and unramified).
 
 Here $\pi(g_\wp)_v$ is the smooth admissible representation of $\textup{GL}_2(\QQ_v)$ %(this is the correct one, Nekovar writes this over F_v, where F is the real field)
  attached to $g_\wp$ at $v$ and $\textup{St}(\mu)$ is the twisted Steinberg, $\mu:K_v^\times \ra \mathbb{C}^\times$ is a character. We refer the reader to \cite[\S12.7.10]{nek} for the precise definition of the \emph{twisted Hida family} and the \emph{twisted cusp form}; the twisted Hida family we are interested in is denoted by $\mathcal{V}$ and a twisted cusp form that the family passes through is denoted by $g_{\mathcal{P}}$ in loc.cit.
  
 Until the end, we assume that ($\mathbf{\textup{H}_v}$) holds true.  See \cite[\S12.3 and \S12.7]{nek} for the content of the assumption  ($\mathbf{\textup{H}_v}$). 
    
  Since we also assumed Hypothesis~\ref{hypo1}, \cite[Th\'eor\`{e}me 7]{mazurtil} shows that 
 $$\frak{h}^{\textup{ord}}\cong \textup{Hom}_{\LL}(\frak{h}^{\textup{ord}},\LL).$$
 It follows then, as explained in~\cite[Proposition 12.7.14.1]{nek} that there is an exact sequence of $\RR[[G_v]]$-modules
  \be \label{seq:TT}0 \lra{F}^+(\TT) \lra \TT \lra{F}^-(\TT)  \lra 0,\ee
  with
  \be\label{eqn:grfactors} {F}^+(\TT)\cong\RR(1)\otimes\mu\,\,\,\, \hbox{ and, }\,\,\,\,   {F}^-(\TT)\cong\RR \ee
  as $\RR[[G_v]]$-modules and $\mu$ is as above.
\begin{rem}
\label{rem:fortheassholeref}
Nekov\'a\v{r}'s description of $\mathbb{V}$ in \cite[Proposition 12.7.17.1]{nek} leads to the exact sequence (\ref{seq:TT}) essentially in the same manner as Proposition~\ref{prop:ordatp} below. This is what we explain in detail in this Remark. Note that one major point for the proof of Proposition~\ref{prop:ordatp} is that the residual representation $\overline{T}$ is $p$-distinguished, in the sense that there is an exact sequence 
$$0\lra\overline{\chi}_1\lra \overline{T} \lra \overline{\chi}_2\lra 0$$
for $\mathcal{R}/\mm_{\mathcal{R}}$-valued characters $\overline{\chi}_1\neq\overline{\chi}_2$ of $\textup{Gal}(\overline{\QQ}_p/\QQ_p)$.

By Proposition 12.7.17.1(i) of \cite{nek}, the sequence 
$$0 \lra{F}^+(\mathbb{V}) \lra \mathbb{V} \lra{F}^-(\mathbb{V})  \lra 0$$
of $\mathcal{L}[G_v]$-modules is exact. Here ${F}^+(\mathbb{V}) \cong \mathcal{L}(1)\otimes\mu$ and $G_v$ acts trivially on the one dimensional $\al$-vector space ${F}^-(\mathbb{V})$. Furthermore, by Proposition 12.7.17.1(ii) of loc.cit., the residual representation $\overline{T}$ fits in an exact sequence 
$$0\lra \omega\otimes\mu \lra \overline{T} \lra \overline{\chi}_0\lra 0$$
of $\mathcal{R}/\mm_{\mathcal{R}}$-vector spaces, where $\omega$ is the mod $p$ cyclotomic character and $\chi_0$ is the trivial character. Thanks to our running assumption that $p\nmid Nv^2-1$, it follows that $ \omega\otimes\mu \neq  \overline{\chi}_0$. Set ${F}^+(\mathbb{T}):={F}^+(\mathbb{V})\cap \mathbb{T}$. We contend to prove that $F^+(\mathbb{T})$ is a  free $\mathcal{R}$-module of rank one. Indeed, note that (since $\TT/F^+(\mathbb{T})$ is $\mathcal{R}$-torsion free) we have a natural injection 
\be\label{eqn:injectioninthefirstcomponent}{F}^+(\mathbb{T})/\mm_{\RR}{F}^+(\mathbb{T})= {F}^+(\mathbb{T})/\left(\mm_{\RR}\TT\cap {F}^+(\mathbb{T})\right)\hookrightarrow \overline{T}\,.\ee
Furthermore, $G_v$ acts on the quotient $F^+(\mathbb{T})/\mm_{\RR}{F}^+(\mathbb{T})$ by $\omega\otimes\mu$, hence (using the fact that $\omega\otimes\mu \neq  \overline{\chi}_0$) it follows that (\ref{eqn:injectioninthefirstcomponent}) induces an injection
$${F}^+(\mathbb{T})/\mm_{\RR}{F}^+(\mathbb{T})\hookrightarrow \mathcal{R}/\mm_R(\omega\otimes\mu)\,.$$
Since ${F}^+(\mathbb{T})\neq 0$, this shows by Nakayama's lemma that
\be\label{eqn:tensoringexactontheleft} 
F^+(\mathbb{T})/\mm_{\RR}{F}^+(\mathbb{T})\stackrel{\sim}{\lra} \mathcal{R}/\mm_R(\omega\otimes\mu)
\ee
 and that the $\RR$-module ${F}^+(\mathbb{T})$ is cyclic and therefore free of rank one. Now set ${F}^-(\mathbb{T}):=\TT/{F}^+(\mathbb{T})$; observe that ${F}^-(\mathbb{T})\hookrightarrow {F}^-(\mathbb{V})$ and hence it is a torsion free $\RR$-module on which $G_v$ acts trivially. By tensoring the exact sequence 
 $$0 \lra{F}^+(\mathbb{T}) \lra \mathbb{T} \lra{F}^-(\mathbb{T})  \lra 0$$
 by $\RR/\mm_\RR$ and using (\ref{eqn:tensoringexactontheleft}), we have the following commutative diagram where all the vertical arrows are isomorphisms:
$$\xymatrix{0 \ar[r]&{F}^+(\mathbb{T})\otimes \mathcal{R}/\mm_R\ar[d] \ar[r] &\mathbb{T}\otimes\mathcal{R}/\mm_R \ar[r]\ar[d]&{F}^-(\mathbb{T})\otimes \mathcal{R}/\mm_R  \ar[r]\ar[d]& 0\\
0\ar[r]& \mathcal{R}/\mm_R(\omega\otimes\mu) \ar[r]& \overline{T} \ar[r]& \mathcal{R}/\mm_R\ar[r]&0}$$
This implies by Nakayama's lemma that $F^-(\mathbb{T})$ is free of rank one as well. This concludes the verification of the asserted filtration (\ref{seq:TT}) of $\TT$.
\end{rem}
 \subsection{The Tamagawa number}
 Let ${\Phi}$ be a finite extension of $\QQ_p$, and let $\frak{O}$ denote its ring of integers. Suppose $T$ is a $\frak{O}[G_v]$-module such that there is an exact sequence of $\frak{O}[G_v]$-modules
 \begin{equation}\label{eqn:T}0\lra\frak{O}(1)\otimes\mu \lra T \lra \frak{O} \lra 0.\end{equation}
  Here, $\mu: K_v^{\times}\ra \{\pm 1\}$ (which may be also thought as a character of the Galois group $G_v$ via local class field theory). Let $c_v=c_v(T)$ denote the $p$-part of the Tamagawa number at $v$, defined as in \cite{FPR91}. In what follows we will compute $c_v$ in terms of the sequence (\ref{eqn:T}). 
 \subsubsection{$\mu \neq \textup{id}$}
 \label{subsec:muneq1}
 Let $L_w$ be the extension of $K_v$ cut by $\mu$, so that $L_w/K_v$ is a quadratic extension. Let $G_w$ denote the absolute Galois group of $L_w$ and $\Delta=\textup{Gal}(L_w/K_v)$. 
 \begin{prop}
 \label{prop:h1vanishes} The following holds:
 \begin{enumerate}
 \item[(i)] $H^1(K_v,\frak{O}(1)\otimes\mu)=0$,
 \item[(ii)] $H^0(I_v,T\otimes\Phi/\frak{O})$ is divisible.
 \item[(iii)] $c_v=1$.
 \end{enumerate}
 \end{prop}
 \begin{proof}
 Using the inflation-restriction sequence and Kummer theory, it follows that $$H^1(K_v, \frak{O}(1)\otimes\mu)=(L_w^\times)^{\mu^{-1}}\otimes_{\ZZ_p}\frak{O}.$$ Here $(L_w^\times)^{\mu^{-1}}$ is the $\mu^{-1}$-part of the $\ZZ_p[\Delta]$-module $L_w^{\times,\wedge}$, where  $L_w^{\times,\wedge}$ is the $p$-adic completion of $L_w^{\times}$. Since $w\nmid p$ and since we assumed~\ref{assume:p}, we have a $\ZZ_p[\Delta]$-equivariant isomorphism $$\textup{ord}_w: L_w^{\times,\wedge} \stackrel{\sim}{\lra} \ZZ_p$$ with the trivial action on $\ZZ_p$. Since $\mu$ is a non-trivial character of $\Delta$, it follows that $(L_w^\times)^{\mu^{-1}}=0$ and thus (i) follows.

 Now taking the $G_v$-cohomology of the sequence (\ref{eqn:T}), using (i) and noting that the $G_v$-action on $\frak{O}\otimes\mu$ is non-trivial, it follows that 
 $$T^{G_v}\stackrel{\sim}{\lra}\frak{O}.$$
 This in return implies that the sequence (\ref{eqn:T}) splits, yielding a decomposition 
 $$T=(\frak{O}(1)\otimes\mu) \oplus \frak{O}$$
  as $\frak{O}[G_v]$-modules. It now follows that $(T\otimes\Phi/\frak{O})^{I_v}\cong \Phi/\frak{O}$ (resp., $(\Phi/\frak{O})^2$) if $\mu$ is ramified (resp., unramified).

Since  $c_v=\# \left(H^0(I_v,T\otimes\Phi/\frak{O})/H^0(I_v,T\otimes\Phi/\frak{O})_{\textup{div}}\right)^{\textup{Fr}_v=1}$, (iii) follows from (ii).
 \end{proof}
 \subsubsection{$\mu=\textup{id}$}
 \label{subsec:mu=id}
 In this case the sequence (\ref{eqn:T}) is 
  \be\label{eqn:Two} 0\lra\frak{O}(1) \lra T \lra \frak{O} \lra 0.\ee
   Let $\sigma=\partial (1)\in H^1(K_v,\frak{O}(1))$ where $\partial: \frak{O} \ra H^1(K_v,\frak{O}(1))$ is the connecting homomorphism in the long exact sequence of the $G_v$-cohomology of the sequence (\ref{eqn:Two}). Kummer theory gives an isomorphism
 $$\textup{ord}_v: H^1(K_v,\ZZ_p(1))\stackrel{\sim}{\lra} K_v^{\times,\wedge}\stackrel{\sim}{\lra} \ZZ_p, $$
  which yields an isomorphism (which we still denote by $\textup{ord}_v$) after tensoring by $\frak{O}$ 
  \be\label{eqn:ord}\textup{ord}_v: H^1(K_v,\frak{O}(1))\stackrel{\sim}{\lra} \frak{O}.\ee
 According to~\cite{cartaneilenberg} pp. 290 and 292, $-\sigma$ is the extension class
of the sequence (\ref{eqn:Two}) inside
$\textup{Ext}^1_{\frak{O}[G_v]}(\frak{O},\frak{O}(1))=H^1(K_v,\frak{O}(1))$. Hence
$\ord_v(\sigma)=0$ if and only if the sequence (\ref{eqn:Two})
splits.
  \begin{prop}
 \label{prop:tamag}
If the sequence (\ref{eqn:Two}) does not split, then  
$$c_v= \# (\frak{O}/\ord_v(\sigma)\frak{O})_{\textup{tors}}.$$
 %In particular, $c_v=1$ if and only if $\ord_v(\sigma)$ is a $p$-adic unit or the sequence~(\ref{eqn:Two}) splits.
 \end{prop}

 \begin{proof}
 By functorialty, we have the following commutative diagram with exact rows:
$$\xymatrix{\frak{O}^{G_v}\ar[d]\ar@{=}[r]&\frak{O}\ar@{=}[d] \ar[r]^(.3){\partial} &H^1(K_v,\frak{O}(1))\ar[d]^{\cong} &&\\
\frak{O}^{I_v}\ar@{=}[r]&\frak{O} \ar[r]^(.3){\partial^\prime} & H^1(I_v,\frak{O}(1)) \ar[r]^(.55){\alpha} &H^1(I_v,T) \ar[r]& H^1(I_v, \frak{O})\cong\ZZ_p
}
$$
The right most vertical arrow is an isomorphism because
\begin{itemize}
\item The cohomological dimension of $G_v/I_v$ is one therefore the map
$$H^1(K_v,\frak{O}(1)) \lra H^1(I_v,\frak{O}(1))$$
is surjective,
\item The group $G_v/I_v$ is pro-cyclic and $\textup{Fr}_v$ is a topological generator, hence for the kernel $H^1(G_v/I_v,\frak{O}(1))$ of this map we have:
$$H^1(G_v/I_v,\frak{O}(1))\cong \frak{O}(1)/(\textup{Fr}_v-1)\frak{O}(1)\cong \frak{O}/(\mathbf{N}v-1)\frak{O},$$
 and since we assumed~(\ref{assume:p}) this is trivial.
\end{itemize}

The second row of the diagram above shows that $H^1(I_v,T)_{\textup{tors}}=\textup{im}(\alpha)_{\textup{tors}}$. Furthermore, since $G_v$ acts trivially on $H^1(I_v,\frak{O}(1))$, it follows that  $\textup{im}(\alpha)_{\textup{tors}}=H^1(I_v,T)_{\textup{tors}}^{\textup{Fr}_v=1}$. Since $c_v=\#H^1(I_v,T)_{\textup{tors}}^{\textup{Fr}_v=1} $, it suffices to prove that 
$$\textup{im}(\alpha)\cong \frak{O}/\ord_v(\sigma)\frak{O}.$$
 The right most vertical isomorphism shows that 
 $$\xymatrix{\textup{im}(\alpha)\cong H^1(K_v,\frak{O}(1))/\partial(\frak{O})\ar[r]_(.65){\ord_v}^(.65){\sim}& \frak{O}/\ord_v(\sigma)\frak{O}}$$
  hence Proposition is proved.

 \end{proof}

 \subsection{The interpolation and the argument}
 \subsubsection{The case $\mu=\textup{id}$}
 \label{subsec:mu=1}
 Let $\frak{p} \subset \RR$ be an arithmetic prime and let $f_\frak{p}$ be the associated modular form attached to $\frak{p}$. Let $\oo(\frak{p})=\RR/\frak{p}$ and $\Phi(\frak{p})=\textup{Frac}\,\oo(\frak{p})$ be its field of fractions. When $\TT$ is as above, the $G_F$-representation $\TT\otimes \Phi(\frak{p})$ is the Galois representation $V(f_\frak{p})$ which is attached to $f_\frak{p}$ by Eichler, Shimura and Deligne~\cite{deligne69}.  We define the specialization map
 $$s_{\frak{p}}:\RR\lra \oo(\frak{p}).$$
 Let $S(\frak{p})$ be the integral closure of $\oo(\frak{p})$ inside $\Phi(\frak{p})$. By slight abuse, we also write $s_\frak{p}$ for the composite
 $$s_\frak{p}: \RR \lra \oo(\frak{p}) \hookrightarrow S(\frak{p}).$$
 The ring $S(\frak{p})$ is a $DVR$. We fix a uniformizer $\pi_\frak{p}$ of $S(\frak{p})$.
%Let $$\xymatrix{x_0:\RR \ar@{->>}[r]&\oo}$$ be an arithmetic specialization (the kernel of which is still denoted by $x_0$). Define $T_n=\TT\otimes_{\RR}\RR/x_0^n$ for $n\geq1$.
We write $T=\TT\otimes S(\frak{p})$ for the $S(\frak{p})$-lattice inside of $V(f_\frak{p})$. Note that $T$ fits in the exact sequence~(\ref{eqn:Two}) and the arguments of \S\ref{subsec:mu=id} apply with $\frak{O}=S(\frak{p})$.  Let $\tau=\ord_v(\sigma) \in S(\frak{p})$, where $\sigma$ is as in \S\ref{subsec:mu=id}. We remark here that we also write $\sigma$ for its image inside $H^1(I_v,S(\frak{p})(1))$ under the isomorphism
$$H^1(K_v,S(\frak{p})(1))\stackrel{\sim}{\lra}H^1(I_v,S(\frak{p})(1))$$
 which appeared in the proof of Proposition~\ref{prop:tamag}.

 Until the end of \S\ref{subsec:mu=1}, suppose the following holds:
 \begin{assume}
 \label{assume:specialtam}
 $\tau \in S(\frak{p})^\times$.
 \end{assume}
 \begin{rem}
 \label{rem:tamtrivial}
 Note that, in view of Proposition~\ref{prop:tamag} and the remark just before Proposition~\ref{prop:tamag}, Assumption~\ref{assume:specialtam} amounts to say that the Hida family $\TT$ should pass through a (twisted) modular form\footnote{In the sense we have explained at the start of Section \ref{subsec:localgaloistame}.} for which the attached Galois representation has the following properties:
  \begin{enumerate}
 \item $p$ does not divide the the Tamagawa number $\textup{Tam}_v(T)$ of the $S(\frak{p})[[G_K]]$-representation $T$ at $v$,
\item the exact sequence (\ref{eqn:Two}) of $S(\frak{p})[[G_v]]$-modules does not split (which is equivalent to saying that $T^{I_v}$ is a free $S(\frak{p})$-module of rank one).
 \end{enumerate}
\end{rem}
 \begin{define}
 Let $$\pmb{\partial}: \RR=\RR^{I_v} \lra H^1(I_v,\RR(1))$$ be the connecting homomorphism of the $I_v$-cohomology of the exact sequence~(\ref{seq:TT}) (with $\mu=1$).
 \end{define}
 Since $I_v$ acts trivially on $\RR$, we have 
 $$H^1(I_v,\RR(1))=H^1(I_v,\ZZ_p(1))\otimes_{\ZZ_p}\RR.$$
  This, together with the isomorphism~(\ref{eqn:ord}) induces  
  \be\label{eqn:ord interpolated} H^1(I_v,\RR(1)) \stackrel{\sim}{\lra} \RR\ee 
  which is in fact an isomorphism of $\RR[[G_v]]$-modules.
 \begin{define}
 Let $\pmb{\tau}\in \RR$ be the image of $\pmb{\partial}(1) \in H^1(I_v,\RR(1))$ under the isomorphism (\ref{eqn:ord interpolated}).
 The element $\pmb{\tau}$ is called the \emph{Tamagawa element} attached to the family $\TT$.
  \end{define}
 \begin{prop}
 \label{prop:unit in RR}
 Suppose Assumption~\ref{assume:specialtam} holds, then $\pmb{\tau} \in \RR^\times$.
 \end{prop}
 \begin{proof}
 We have a commutative diagram\\\\
 $$ \xymatrix@C=.29in@R=.2in{1 \ar @{|->}[dd] \ar @/^2pc/@{|->}[rrrrrr]&**[l]\in&\RR \ar[dd]^{s_\frak{p}}\ar[r]^(.34){\pmb{\partial}}& H^1(I_v,\RR(1)) \ar[dd]^{s_\frak{p}} \ar[r]^(.68){\sim}&\RR\ar[dd]^{s_\frak{p}}&**[l] \ni&\pmb{\tau} \ar @{|->}[dd]\\
&&&&&&\\
1\ar @/_2pc/@{|->}[rrrrrr]&**[l]\in&S(\frak{p})\ar[r] & H^1(I_v,S(\frak{p})(1))\ar[r]^(.67){\sim}&S(\frak{p}) &**[l] \ni &\tau
 }$$%\\\\ (recall the isomorphism $\RR/x_0 \stackrel{\sim}{\lra} \oo$ induced by the specialization $x_0$). Assumption~\ref{assume:specialtam} says that $\tau \in \oo^\times$. This together with the diagram above, our assumption that $\RR$ is a power series ring in one variable and Weirstrass preparation shows that $\pmb{\tau}\in \RR^\times$, as desired.
 \\\\
Let $\bar{\tau}$ be the image of $\tau$ under the map
 $$S(\frak{p}) \lra k(\frak{p}):=S(\frak{p})/\pi_\frak{p}S(\frak{p}).$$
 By our assumption that $\tau \in S(\frak{p})^\times$, it follows that $\bar{\tau} \neq 0$.  Furthermore, there is a commutative diagram
 $$\xymatrix@C=.15in@R=.15in{\pmb{\tau}\ar @/_4pc/@{|->}[ddrrrrrrr] &\in&\RR\ar@{-{>>}}[rdd]\ar[rr]^{s_\frak{p}}&&\oo(\frak{p})\ar@{^{(}->}[rr]&&S(\frak{p})\ar@{-{>>}}[ddl]\\
&&&&&&&\\
&&&\RR/\frak{m}\ar@{^{(}->}[rr]&&k(\frak{p})&\ni&\bar{\tau}
} $$\\\\
where the injection $\RR/\frak{m} \hookrightarrow k(\frak{p})$ is because $\RR/\frak{m}$ is the unique field that $\RR$ surjects onto. This shows that the image of $\tau$ under the natural map $\RR\ra\RR/\frak{m}$ is non-zero, and Proposition follows.
 \end{proof}
 %What is excluded below concerns the case $\RR=\oo[[\Gamma]]$.
% Let $\textup{Tam}_v(T_n)$ denote the Tamagawa number at $v$ for the $\oo[[G_F]]$-representation $T_n$. Note that $T_n$ is a free $\oo$-module of finite rank and we use the definition of \cite{FPR91} of Tamagawa factors.
 %\begin{prop}
 %Under Assumption~\ref{assume:specialtam},   $p$ does not divide $\textup{Tam}_v(T_n)$.
 %\end{prop}
 %\begin{proof}
 %The square in the following diagram commutes and the second row is exact:\\\\
 %$$\xymatrix@C=.115in{1\ar@{|->}[d]\ar @/^2.5pc/@{|->}[rrrrr] &\in& \RR\ar[d] \ar[r]^(.26){\pmb{\partial}}&H^1(I_v,\RR(1))\cong\RR\ar[d]&**[r]\ni& **[l]\pmb{\tau}\\
 %1 &\in& \RR/x_0^n \ar[r]_(.29){\partial_n}&H^1\left(I_v,\left(\RR/x_0^n\right)(1)\right)\ar[r]&H^1(I_v,T_n) \ar[r]&H^1(I_v,\RR/x_0^n)\ar[r]&0
 %}
%$$

%Since $\pmb{\tau}\in \RR^\times$ by Proposition~\ref{prop:unit in RR}, it follows at once that the map $\partial_n$ is surjective hence the second row of the diagram reduces to the isomorphism of $\oo$-modules:
%$$\xymatrix{H^1(I_v,T_n)\ar[r]^(.45){\sim} &H^1(I_v,\RR/x_0^n).}$$
%The $\oo$-module $H^1(I_v,\RR/x_0^n)\cong \RR/x_0^n$ is $\oo$-torsion free, hence $H^1(I_v,T_n)$ is also $\oo$-torsion free. This completes the proof of the Proposition.
 %\end{proof}

Let now
$$s: \RR \lra S$$
 be any specialization, where $S$ is a discrete valuation ring with uniformizer $\pi_S$. We set $T_S=\TT\otimes_{\RR}S$. Let $\textup{Tam}_v(T_S)$ denote the Tamagawa number at $v$ for the $S[[G_F]]$-representation $T_S$. Note that $T_S$ is a free $S$-module of rank two and as remarked above, we use the definition of \cite{FPR91} of Tamagawa factors.
\begin{prop}
\label{prop:tamuniformity}
 Suppose that Assumption~\ref{assume:specialtam} above holds true. Then $\textup{Tam}_v(T_S)$ is prime to $p$.
\end{prop}
\begin{proof}
As in the proof of Proposition~\ref{prop:unit in RR}, there is a commutative diagram\\
 $$ \xymatrix@C=.29in@R=.2in{1 \ar @{|->}[dd] \ar @/^2pc/@{|->}[rrrrrr]&**[l]\in&\RR \ar[dd]^{s}\ar[r]^(.34){\pmb{\partial}}& H^1(I_v,\RR(1)) \ar[dd]^{s} \ar[r]^(.68){\sim}&\RR\ar[dd]^{s}&**[l] \ni&\pmb{\tau} \ar @{|->}[dd]\\
&&&&&&\\
1\ar @/_2pc/@{|->}[rrrrrr]&**[l]\in&S\ar[r] & H^1(I_v,S(1))\ar[r]^(.67){\sim}&S&**[l] \ni &\tau_S
 }$$
 \\\\\\
 Also, the commutative diagram\\
 $$\xymatrix@C=.2in{ \pmb{\tau}\ar @/^3.1pc/@{|->}[drrrrrr] \ar@{|->}[d]& \in & \RR\ar[rr]^{s}\ar[d]&&S\ar[d]&&\\
 \bar{\pmb{\tau}}\ar @/_1.5pc/@{|->}[rrrrrr]&\in&\RR/\frak{m}\ar@{^{(}->}[rr]&&S/\pi_SS&\ni&\overline{\tau_S}
 }$$\\\\
 shows that $\tau_S \in S^\times$, since $\bar{\pmb{\tau}}\neq 0$ by Proposition~\ref{prop:unit in RR}. This completes the proof by making use of Proposition~\ref{prop:tamag} with $\frak{O}=S$ and $T=T_S$.
\end{proof}
 \subsubsection{The case $\mu\neq\textup{id}$}
 The case $\mu\neq\textup{id}$ (in (\ref{seq:TT})) is handled just as in \S\ref{subsec:muneq1}. Following the proof of Proposition~\ref{prop:h1vanishes}, one is able to prove:
 \begin{prop}
 \label{prop:tamtrivmuneq1}
 Suppose $T_S$ is as in~\S\ref{subsec:mu=1}. Then under Assumption~\ref{assume:specialtam}, $p \nmid \textup{Tam}_v(T_S)$.
 \end{prop}
 The following Lemma will be crucial when checking the local properties of the big Heegner point Euler system.
\begin{lemma}
\label{lemma:torsionfree}
Let $\frak{L}$ be any unramified $p$-extension of $K_v$. Then under Assumption~\ref{assume:specialtam}, the $\RR$-module $H^1(\frak{L}^{\textup{ur}},\TT)$ is torsion-free.
\end{lemma}
\begin{proof}
As the Assumption ($\mathbf{\textup{H}_v}$) is in effect, the $G_v$-representation $\TT$ fits in an exact sequence
$$0 \lra \RR(1)\otimes\mu \lra \TT \lra R \lra 0,$$
where the character $\mu$ is described above.

In the case $\mu\neq \textup{id}$, the proof of Proposition~\ref{prop:h1vanishes} shows that $\TT=(\RR(1)\otimes\mu) \oplus \RR$ as $G_v$-modules. Thus, we need to verify that the $\RR$-modules $H^1(\frak{L}^{\textup{ur}}, \RR(1)\otimes\mu)$ and $H^1(\frak{L}^{\textup{ur}}, \RR)$ are torsion-free. For the module $$H^1(\frak{L}^{\textup{ur}}, \RR)\cong \textup{Hom}(\textup{Gal}(\overline{K}_v/\frak{L}^{\textup{ur}}),\RR)$$
this is clear. Also, when $\mu$ is an unramified character of $G_v$, it follows that the inertia acts trivially on  $\RR(1)\otimes\mu$ and hence $$H^1(\frak{L}^{\textup{ur}},\RR(1)\otimes\mu)\cong  \textup{Hom}(\textup{Gal}(\overline{K}_v/\frak{L}^{\textup{ur}}),\RR(1)\otimes\mu)\cong \textup{Hom}(\textup{Gal}(\overline{K}_v/\frak{L}^{\textup{ur}}),\RR)$$ is free as well. When $\mu$ is ramified, the argument carries over after replacing $\frak{L}^{\textup{ur}}$ by a quadratic extension.
%The inflation-restriction sequence, along with the fact that the cohomological dimension of $\textup{Gal}(\frak{L}^{\textup{ur}}/\frak{L})$ is one, yields an exact %sequence
%$$0\lra H^1(\frak{L}^{\textup{ur}}/\frak{L},\ZZ_p(1)\otimes\mu) \lra $$

In the case $\mu=\id$, we have the following exact sequence as in the proof of Proposition~\ref{prop:tamag}:

$$\xymatrix{%H^0(\frak{L},\RR)\ar[d]\ar@{=}[r]&\RR\ar@{=}[d] \ar[r]^(.3){\partial} &H^1(\frak{L},\RR(1))\ar[d]^{\cong} &&\\
 \RR=H^0(\frak{L}^{\textup{ur}},\RR) \ar[r]^(.53){\partial} & H^1(\frak{L}^{\textup{ur}},\RR(1)) \ar[r]^(.55){\alpha} &H^1(\frak{L}^{\textup{ur}},\TT) \ar[r]& H^1(\frak{L}^{\textup{ur}}, \RR)\cong\RR
}
$$
%The right most vertical arrow is an isomorphism because
%\begin{itemize}
%\item $G_\frak{L}$ acts trivially on $H^1(\frak{L}^{\textup{ur}},\RR(1))$.
%\item The cohomological dimension of $\textup{Gal}(\frak{L}^{\textup{ur}}/\frak{L})$ is one, therefore the map
%$$H^1(\frak{L},\RR(1)) \lra H^1(\frak{L}^{\textup{ur}},\RR(1))$$
%is surjective.
%\item Let $w$ be the maximal ideal of $\frak{L}$. Then the group $\textup{Gal}(\frak{L}^{\textup{ur}}/\frak{L})$ is pro-cyclic and $\textup{Fr}_w$ is a topological generator, hence for the kernel $H^1(\frak{L}^{\textup{ur}}/\frak{L},\RR(1))$ of this map we have:
%$$H^1(\frak{L}^{\textup{ur}}/\frak{L},\RR(1))\cong \RR(1)/(\textup{Fr}_w-1)\RR(1)\cong \RR/(\mathbf{N}w-1)\RR,$$
 %and since we assumed~\ref{assume:p} and $\LL/K_v$ is a $p$-extension, this is trivial.
%\end{itemize}
Since $\frak{L}/K_v$ is unramified, the proof of Proposition~\ref{prop:unit in RR} shows that the map $\partial$ is surjective under Assumption~(\ref{assume:specialtam}), hence the map $\alpha$ is injective. The proof is now complete.
\end{proof}
\section{Kolyvagin descent for the big Heegner points}
\label{sec:KSdescent}
\subsection{Selmer groups}
\subsubsection{Local conditions and Selmer structures} Throughout this section, let $L$ be a finite extension of $K$ and for each prime $v$ of $L$ define $L_v^{\textup{unr}}$ as the maximal unramified extension of $L_v$. %Fixing a prime $\bar{v}$ of $\overline{Q}$ above $v$,
Let $\mathcal{I}_v\subset \mathcal{D}_v$ be a fixed choice of inertia and decomposition groups of $v$. Let $R$ be any local Noetherian ring and $M$ any $R[[G_{K}]]$-module.
\begin{define}
\label{selmer structure}
A \emph{Selmer structure} $\FF$ on $M$ is a collection of the following data:
\begin{itemize}
%\item a finite set $\Sigma(\FF)$ of places of $L$, including all infinite places and primes above $p$, and all primes where $M$ is ramified.
\item For every $v\mid Np$, choose a local condition on $M$ (which we view now as a ${R}[[\mathcal{D}_{v}]]$-module), i.e., a choice of ${R}$-submodule
 $$H^1_{\FF}(L_{v},M) \subset H^1(L_{v},M).$$
\item For $v \nmid Np$, set
$$H^1_{\FF}(L_{v},M)=H^1_{f}(L_{v},M):=\ker\left(H^1(L_v,M)\ra H^1(L_v^{\textup{unr}},M)\right)$$

%, where the module $H^1_{{f}}(k_{\lambda},M)$ is the \emph{finite} part of $H^1(k_{\lambda},M)$, defined as in~\cite[Definition 1.1.6]{mr02}.

 \end{itemize}
\end{define}

\begin{define}
 The \emph{semi-local cohomology group} at a rational
prime $\ell$ is defined by setting
$$H^i(L_\ell,M):=\bigoplus_{v|\ell} H^i(L_v,M),$$
where the direct sum is over all primes $v$ of $L$ lying above
$\ell$.
\end{define}
\begin{define}
\label{selmer group}
If $\FF$ is a Selmer structure on $M$, we define the \emph{Selmer module} $H^1_{\FF}(L,M)$ as
 $$H^1_{\FF}(L,M):=\ker\left(H^1(L,M) \lra \prod_v H^1(L_v,M)/H^1_{\FF}(L_v,M)\right).$$
 \end{define}

\subsubsection{Greenberg Conditions}
Let $\TT$ be the big self-dual Galois representation attached to a Hida family.
\begin{prop}
\label{prop:ordatp}
Suppose $v$ is any place of $L$ above $p$. Then there is an exact sequence of $\RR[[D_v]]$-modules
\be\label{seq:ordatp}0\lra \textup{F}^+_v(\TT)\lra \TT \lra \textup{F}^-_v(\TT)\lra 0\ee
 such that both $\textup{F}^+_v(\TT)$ and $\textup{F}^-_v(\TT)$ are free of rank one over $\RR$.
\end{prop}
See \cite[Prop. 2.4.1]{howard} (and \cite[\S12.7.8-10]{nek} when dealing with a Hida family of Hilbert modular forms) for a proof of this statement.

For any ring homomorphism $\frak{s}:\RR\ra S$ (where $S$ is a local Noetherian ring), set $T_S=\TT\otimes_\frak{s}S$. By tensoring the sequence (\ref{seq:ordatp}) by S, we also define $\textup{F}^\pm_v(T_S)$ for any of the modules $T_S$ above. Furthermore, if $S$ is a ring which is finitely generated as an $\oo$-module, set $V_S=T_S\otimes_{\oo}E$ and define $\textup{F}^\pm_v(V_S)$ in a similar manner.
\begin{define}
\label{def:grlocal}
\begin{itemize}
\item[(i)] The \emph{strict Greenberg Selmer structure} $\FF_{\textup{Gr}}$ on $\TT$ by setting local conditions as
$$H^1_{\FF_{\textup{Gr}}}(L_v,\TT)=\left\{\begin{array}{ccr}
\ker\left(H^1(L_v,\TT)\lra H^1(L_v^{\textup{unr}},\TT)\right)&,& \hbox{ if } v\nmid p\\
\ker\left(H^1(L_v,\TT)\lra H^1(L_v,\textup{F}^-(\TT))\right)&,& \hbox{ if } v\mid p
\end{array}
\right.$$
\item[(ii)] Let $T$ be any subquotient of $\TT$. Then let  $\FF_{\textup{Gr}}$ on $T$ be the Selmer structure defined by propagating the Selmer structure $\FF_{\textup{Gr}}$ on $\TT$ via \cite[Example 1.1.2]{mr02}.
\item[(iii)] Let $S$ be a ring for which $T_S$ and $V_S$ is defined. We define a Selmer structure $\tilde{\FF}_{\textup{Gr}}$ on $T_S$ by setting
$$H^1_{\tilde{\FF}_{\textup{Gr}}}(L_v,T_S)=\left\{\begin{array}{ccr}
\ker\left(H^1(L_v,T_S)\ra H^1(L_v^{\textup{unr}},V_S)\right)&,& \hbox{ if } v\mid N\\
H^1_{\FF_{\textup{Gr}}}(L_v,T_S)&,& \hbox{ if } v\nmid N
\end{array}
\right.$$
\end{itemize}
\end{define}

\begin{prop}
\label{prop:comparegrwithtildegrwithbk}
Suppose the Assumption~\ref{assume:p} and Assumption~\ref{assume:specialtam} holds; see Remark~\ref{rem:tamtrivial} for the content of the latter assumption. Then
$$H^1_{\tilde{\FF}_{\textup{Gr}}}(L_v,T_S)=H^1_{\FF_{\textup{Gr}}}(L_v,T_S)$$
for all $v$.
\end{prop}
\begin{proof}
This follows from Proposition~\ref{prop:tamuniformity} and \cite[Lemma 3.5]{r00}.
\end{proof}
\subsection{Kolyvagin systems for $\TT$ and $\TT\otimes\LL^{\textup{ac}}$}
\label{subsec:KSforTT}
Let $K_\infty$ be the anticyclotomic $\ZZ_p$-extension of $K$ and $\Gamma^{\textup{ac}}=\textup{Gal}(K_\infty/K)$, $\LLac=\ZZ_p[[\Gamma^{\textup{ac}}]]$. Let $\gamma \in \Gamma^{\textup{ac}}$ be a fixed topological generator, and let $\{\pi,x\}$ be a maximal regular sequence for the two-dimensional Gorenstein ring $\RR$, where $\pi=\pi_E$ is the uniformizer of $E$ fixed in the introduction. For each $k,m,r \in \ZZ^+$, let $R_{k,m}$ (resp., $R_{k,m,r}$) be the artinian local ring $\RR/(\pi^k,x^m)$ (resp., $R_{k,m}\otimes_{\ZZ_p}\LLac/(\gamma-1)^r$) and $T_{k,m}$ (resp., $T_{k,m,r}$) be the module $\TT\otimes_{\RR}R_{k,m}$ (resp., $\TT\otimes_{\RR}R_{k,m,r}$). Note that we allow $G_K$ act on both factors defining $T_{k,m,r}$ via the map $G_K\twoheadrightarrow \Gamma^{\textup{ac}}$. Let $\PP$ be the set of all primes $\lambda\subset \oo_K$ of degree $2$ such that $\lambda\nmid Np$ and define $\PP_{k,m,r}\subset \PP$ to be the collection of primes $\lambda$ which satisfy:
\begin{itemize}
\item[(i)] Let $\ell$ be the rational prime below $\lambda$. Then $\ell+1 \equiv 0 \mod p^k$.
\item[(ii)] Let $D_\lambda \subset G_K$ be any decomposition group for the prime $\lambda$. Then
$$D_\lambda \subset \ker (G_K\ra \textup{Aut}_{\RR}(T_{k,m,r})).$$
Note that this condition is independent of the choice of the decomposition group $D_\lambda$ and implies that $\textup{Fr}_\lambda$ acts trivially on $T_{k,m,r}$.
\end{itemize}
For $\lambda \in \PP_{k,m,r}$, set
$$H^1_{f}(K_\lambda,T_{k,m,r}):=\ker (H^1(K_\lambda,T_{k,m,r})\lra H^1(K_\lambda^{\textup{unr}},T_{k,m,r})),$$
and
$$H^1_{s}(K_\lambda,T_{k,m,r}):= H^1(K_\lambda,T_{k,m,r})/ H^1_f(K_\lambda,T_{k,m,r}).$$
\begin{prop}
\label{prop:fscomparison}
Suppose $\lambda \in \PP_{k,m,r}$. Let $\textup{k}_\lambda=\oo_K/\lambda$  and $\textup{k}_\ell=\ZZ/\ell\ZZ$. Then there is a finite-singular comparison map
$$\phi^{\textup{fs}}_\lambda: H^1_{f}(K_\lambda,T_{k,m,r})\stackrel{\sim}{\lra}H^1_{s}(K_\lambda,T_{k,m,r})\otimes \textup{k}_\lambda^\times/\textup{k}_\ell^\times.$$
\end{prop}
\begin{proof}
It follows from \cite[Lemma 1.2.1]{mr02} and our assumption $\lambda \in  \PP_{k,m,r}$ that
$$H^1_{f}(K_\lambda,T_{k,m,r})\stackrel{\sim}{\lra}T_{k.m,n} \stackrel{\sim}{\longleftarrow}H^1_{s}(K_\lambda,T_{k,m,r})\otimes\textup{k}_\lambda^\times.$$
Identifying the $p$-Sylow subgroups of $\textup{k}_\lambda^\times$ and $\textup{k}_\lambda^\times/\textup{k}_\ell^\times$ the Proposition follows.
\end{proof}
\begin{define}
\label{def:Gs}
\begin{itemize}
\item[(a)] For $\lambda \in \PP$ and $\ell$ the prime below $\lambda$, define $\mathcal{G}(\ell)=\textup{k}_\lambda^\times/\textup{k}_\ell^\times$.
\item[(b)] Let $\NN_{k,m,r}$ (resp., $\NN$) be the set of square-free products of the rational primes $\ell$ that lie above the primes chosen among of $\PP_{k,m,r}$ (resp., $\PP$).
\item[(c)] For $n \in \NN$, define $\mathcal{G}(n)=\bigotimes_{\ell|n}\mathcal{G}(\ell)$.
\end{itemize}
\end{define}

\begin{define}
\label{def:transverse} For $\lambda \in \PP$ and $\ell$ the prime
below it, let $H_\ell$ be the ring class field of conductor $\ell$.
Since $\lambda$ splits completely in the Hilbert class field of $K$,
the maximal $p$-subextension $L$ of the extension
$(H_{\ell})_\lambda /K_\lambda$ is totally ramified abelian
$p$-extension of $K_\lambda$. Furthermore, its Galois group is
canonically identified with the $p$-Sylow subgroup of $G_\lambda$ by
class field theory hence it is also the maximal totally tamely
ramified abelian $p$-extension of $K_\lambda$. Define the
\emph{transverse submodule}
$$H^1_{\textup{tr}}(K_\lambda,M)=\ker (H^1(K_\lambda,M)\lra H^1(L,M))$$
for any $G_{K_\lambda}$-module $M$.
\end{define}
\begin{lemma}
\label{lemma:trsing}
The transverse submodule $H^1_{\textup{tr}}(K_\lambda,T_{k,m,r})$ projects isomorphically onto the singular quotient $H^1_{s}(K_\lambda,T_{k,m,r})$ under the natural projection
$$H^1(K_\lambda,T_{k,m,r})\lra H^1_{s}(K_\lambda,T_{k,m,r}).$$
\end{lemma}
This is \cite[Lemma 1.2.4]{mr02}.

\begin{define}
\label{def:modifiedSelmer}For a Selmer structure $\FF$ on  a $G_K$-representation $M$ and $n \in \NN$, define the modified Selmer structure $\FF(n)$ on $M$ by setting
$$H^1_{\FF(n)}(K_v,M)=\left\{\begin{array}{ccr}
H^1_{\FF}(K_v,M)&,& \hbox{ if } v\nmid n\\
H^1_{\textup{tr}}(K_v,M)&,& \hbox{ if } v \mid n
\end{array}\right.$$
\end{define}

\begin{define}
\label{def:bigKS} The $\RR\otimes\LL^{\textup{ac}}$-module of big
Kolyvagin systems for $\TT\otimes\LL^{\textup{ac}}$ is defined as
$$\overline{\textbf{KS}}(\TT\otimes\LL^{\textup{ac}},\FF_{\textup{Gr}}):=\varprojlim \textbf{KS}(T_{k,m,r}, \FF_{\textup{Gr}},\PP_{k,m,r}),$$
where each of the modules $\textbf{KS}(T_{k,m,r},
\FF_{\textup{Gr}},\PP_{k,m,r})$ of Kolyvagin systems over the
artinian ring $R_{k,m,r}$ is defined following \cite[Definition
3.1.3]{mr02}, via the constructions given above.
\end{define}
\subsection{Big Heegner point Kolyvagin system}
\label{sec:EStoKS}

We start this section by recalling Kolyvagin's derivative
construction. Let $H_c$ denote the ring class field of $K$ of
conductor $c$, and for $c$ prime to $p$, let $K(c)$ be the maximal
$p$-extension in $H_c/K$. We assume until the end that
\begin{itemize}
\item The class number of $K$ is prime to $p$ (equivalently, $K(1)=K$).
\item $\overline{T}$ is an absolutely irreducible $G_K$-representation.
\item The twisted Hida family passes through a member for which all the Tamagawa factors at primes $\ell$ dividing the tame conductor $N$ are prime to $p$.
\end{itemize}
Then the maximal $p$-extension
$K_\alpha$ in $H_{p^{\alpha+1}}$ satisfies that
$[K_\alpha:K]=p^{\alpha}$. For $(c,p)=1$, write $K_\alpha(c)$ for
the composite field of $K(c)$ and $K_\alpha$.
%\subsubsection{Local properties of the big Euler system}
For Howard's big Euler system
$\{\frak{X}_{cp^{\alpha}}\}_{c,\alpha}$ of Heegner points defined as in \cite[\S2.2]{howard}, we set using \cite[Prop. 2.3.1]{howard}
$$\frak{z}_{c,\alpha}:=\textup{cor}_{H_{_{cp^{\alpha+1}}}/K_{\alpha}(c)}\,U_p^{-\alpha}\frak{X}_{cp^{\alpha+1}} \in H^1(K_\alpha(c),\mathbb{T})$$
for every $c$ prime to $Np$, where $U_p$ is the Hecke operator.
\begin{rem}
\label{rem:maxunr}
Let $\lambda$  a place of $K$ and suppose $L/K_\lambda$ is an unramified extension. Then $L^{\textup{ur}}=K_\lambda^{\textup{ur}}$, as the field $L^{\textup{ur}}$ is unramified over $K_\lambda$, and the composite $K_\lambda^{\textup {ur}}L$ is unramified over $L$.
\end{rem}

\begin{prop}
\label{prop:localbigEuler} Under the running hypotheses,
$$\frak{z}_{c,\alpha}\in H^1_{\FF_{\textup{Gr}}}(K_\alpha(c),\TT).$$
\end{prop}
\begin{proof}
We need to check that $\frak{z}_{c,\alpha}\in H^1_{\FF_{\textup{Gr}}}(K_\alpha(c)_v,\TT)$ for every place $v$ of $K_\alpha(c)$.

Suppose that $w| v| \lambda|N$, where $w$ (resp., $v$, resp., $\lambda$) is a prime of $H_{_{cp^{\alpha+1}}}$ (resp., of $K_\alpha(c)$, resp., of $K$). For ease of notation, let $\frak{L}=(H_{_{cp^{\alpha+1}}})_w$ and $\frak{K}=K_\alpha(c)_v$. The proof of \cite[Prop. 2.4.5]{howard} shows that the restriction of $\frak{X}_{cp^{\alpha}}$ to $H^1(\frak{L}^{\textup{ur}},\TT)$ is $\RR$-torsion, which is trivial by Lemma~\ref{lemma:torsionfree}. This completes the proof that  $\textup{loc}_w(\frak{X}_{cp^{\alpha}}) \in H^1_{\FF_{\textup{Gr}}}(\frak{L},\TT)$. We have a commutative diagram
$$\xymatrix{
H^1(\frak{L},\TT)\ar[r]^{\textup{res}}\ar[d]^{\textup{cores}} &H^1(I_\lambda,\TT)^{G_{\frak{L}}}\ar[d]^{\textup{N}_{\frak{K}/\frak{L}}} \\
H^1(\frak{K},\TT)\ar[r]^{\textup{res}}& H^1(I_\lambda,\TT)^{G_\frak{K}}
}$$
where we use Remark~\ref{rem:maxunr} to identify $I_\lambda$ with the Galois groups of the extensions $\overline{K}_\lambda/\frak{K}^{\textup{ur}}$ and $\overline{K}_\lambda/\frak{L}^{\textup{ur}}$. Since
the image of $\textup{loc}_w(\frak{X}_{cp^{\alpha}})$ under the left vertical map is $\textup{loc}_v(\frak{z}_{c, {\alpha}})$, and  its image under the upper horizontal map is trivial, it follows that $\textup{loc}_v(\frak{z}_{c, {\alpha}}) \in  H^1_{\FF_{\textup{Gr}}}(\frak{K},\TT)$ as desired.

For a prime $w\nmid Np$  of $H_{_{cp^{\alpha+1}}}$, Howard in \cite[Prop. 2.4.5]{howard} proves that $\textup{loc}_w(\frak{X}_{cp^{\alpha}}) \in H^1_{\FF_{\textup{Gr}}}(\frak{L},\TT)$ and the proposition follows for every $v\nmid N$ as above.

Finally, for $w\mid p$ of $H_{_{cp^{\alpha+1}}}$, Howard in loc. cit. shows that $\textup{loc}_w(\frak{X}_{cp^{\alpha}}) \in H^1_{\FF_{\textup{Gr}}}(\frak{L},\TT)$ and the proposition follows from the commutative diagram:
$$\xymatrix{
H^1(\frak{L},\TT)\ar[r]\ar[d] &H^1(\frak{L},\textup{F}_v^-(\TT))\ar[d]\\
H^1(\frak{K},\TT)\ar[r]& H^1(\frak{K},\textup{F}_v^-(\TT))
}$$
\end{proof}
Let 
$$\mathcal{A}_s=\ker\left(\oo[[\textup{Gal}(H_{p^{\infty}}/H_{p^s})]] \lra \oo \right)$$ 
be the augmentation ideal of $\oo[[\textup{Gal}(H_{p^{\infty}}/H_{p^s})]]$.
Until the end of this section fix $k,m,r \in \ZZ^+$ as well as a positive
integer $s$ for which the following condition holds: 
\begin{equation}
\label{eqn:choice_np}
\hbox{The image of } \mathcal{A}_s \hbox{ is contained in the ideal }  (\pi^k,
(\gamma-1)^{p^r}) \hbox{ of } \LL^{\textup{ac}}
\end{equation}
under the map induced from the natural inclusion $\textup{Gal}(H_{p^{\infty}}/H_{p^s}) \subset \Gamma^{\textup{ac}}.$

\begin{define}
\label{def:kolder}
Fix a prime $\lambda \in \PP_{k,m,r}$ and let $\ell$ be the rational prime below $\lambda$.
\begin{itemize}
\item[(i)]  Let $\mathcal{G}_\ell=\textup{Gal}(K(\ell)/K)$. Note then that $\mathcal{G}_\ell$ is the $p$-Sylow subgroup of the cyclic group $\mathcal{G}(\ell)=k^\times_\lambda/k_\ell^\times$ defined above. Let $\sigma_\ell$ be the generator of $\mathcal{G}_\ell$.
\item[(ii)] For a square free integer $n \in \NN_{k,m,r}$, we set $\mathcal{G}_n={\displaystyle \bigotimes_{\ell|n}\,\mathcal{G}_\ell}$ and define $G(n)={\displaystyle\prod_{\ell|n}} \mathcal{G}_\ell$. Then for $m|n$,
    $$\textup{Gal}(K(n)/K(m))\cong{\displaystyle\prod_{\ell|n}}\mathcal{G}_\ell\cong G(n/m).$$
    Since we assumed that $p$ is prime to the class number of $K$, we also have that $$G(n)\cong\textup{Gal}(K(n)/K).$$%This agrees with Howard's (Heegner point Kolyvagin system) definition as long as we assume that p is prime to the class number of K.
    \item[(iii)] $\displaystyle\mathcal{D}_\ell=\sum_{i=0}^{|\mathcal{G}_\ell|-1}i\sigma_\ell^{i}\in \ZZ_p[\mathcal{G}_\ell]$ and $D_n={\displaystyle\prod_{\ell|n}D_\ell} \in \ZZ_p[G(n)]$.
\end{itemize}
\end{define}
\begin{define}
\label{def:KSkmr1}
For $n\in \NN_{k,m,r}$ and $\alpha\geq s$, define $\frak{z}_{n,\alpha}^{\prime}=D_n\,\frak{z}_{n,\alpha}\, \in H^1(K_\alpha(n),\TT)$.
\end{define}
By the standard telescoping identity satisfied by the derivative operators $D_n$, it follows for $n \in \NN_{k,m,r}$ that the image of $\frak{z}_{n,\alpha}^{\prime}$ (which we denote by $\kappa_{[n,\alpha]}^{\prime}$) under the reduction map
$$H^1(K_\alpha(n),\TT)\lra H^1(K_\alpha(n),T_{k,m})$$
lies inside $H^1(K_\alpha(n),T_{k,m})^{G(n)}$. On the other hand, since $G(n)$ is generated by the ($p$-parts of the) inertia groups at the primes dividing $n$ and $T_{k,m}$ is unramified at these primes, it follows that
$$H^0(K(n),T_{k,m})=H^0(K,T_{k,m}).$$
Furthermore, since we assumed that the $G_K$-representation $\overline{T}$ is irreducible, we have that $H^0(K,T_{k,m})=0$. The restriction map
\be\label{ref:eqnisomuseful}H^1(K_\alpha,T_{k,m})\lra H^1(K_\alpha(n),T_{k,m})^{G(n)}\ee
is therefore an isomorphism.

\begin{define}
\label{def:KSkmr2}
\begin{itemize}
\item[(i)] For $n\in \NN_{k,m,r}$ and $\alpha\geq s$, define $\kappa_{[n,\alpha]}$ as the inverse image of $\kappa_{[n,\alpha]}^{\prime}$ under the isomorphism (\ref{ref:eqnisomuseful}).
\item[(ii)] Let $\kappa_n \in H^1(K,T_{k,m,r})$ be the image of $\kappa_{[n,\alpha]}$ under the map
$$H^1(K_\alpha,T_{k,m}) \lra H^1(K,T_{k,m,r})$$
induced by Shapiro's Lemma, and the choice of $s$ as in (\ref{eqn:choice_np}). It is not hard to see that the definition of $\kappa_n$ does not depend on the choice of $\alpha$ and $s$, thanks to the norm compatibility of the classes $\frak{z}_{c,\alpha}$ as $\alpha$ varies.
\end{itemize}
\end{define}
Recall that $\frak{z}_{n,\alpha} \in H^1_{\FF_{\textup{Gr}}}(K_\alpha(n),\TT)$ by Proposition~\ref{prop:localbigEuler}.
\subsubsection{Local properties away from $Np$}
\begin{prop}
\label{prop:localgoodplaces}
For a place $v\nmid Nnp$ of $K$ and a place $w$ of $K_\alpha$ above $v$, we have
\begin{itemize}
\item[(i)] $\textup{loc}_w(\kappa_{[n,\alpha]}) \in H^1_{\textup{ur}}(K_\alpha)_w,T_{k,m}):=\ker(H^1((K_\alpha)_w,T_{k,m})\lra H^1((K_\alpha)_w^{\textup{ur}},T_{k,m}))$,
\item[(ii)] $\textup{loc}_v(\kappa_{n}) \in H^1_{\textup{ur}}(K_v,T_{k,m,r}):=\ker(H^1(K_v,T_{k,m,r})\lra H^1(K_v^{\textup{ur}},T_{k,m,r}))$.
\end{itemize}
\end{prop}
\begin{proof}
Let $w^\prime$ be a place of $K_\alpha(n)$ above $w$. As remarked above, we have $K_\alpha(n)_{w^\prime}^{\textup{ur}}=(K_\alpha)_w^{\textup{ur}}=K_v^{\textup{ur}}.$ Using the fact that
$\textup{loc}_{w^\prime}(\frak{z}_{n,\alpha})$ (and therefore $\textup{loc}_{w^\prime}(\frak{z}_{n,\alpha}^{\prime})$ as well) lies in
$$H^1_{\textup{ur}}(K_\alpha(n)_{w^\prime},\TT)=\ker(H^1(K_\alpha(n)_{w^\prime},\TT)\lra H^1(K_v^{\textup{ur}},\TT)),$$ the diagram below with commutative squares proves (i):
$$\xymatrix{H^1((K_\alpha)_w,T_{k,m})\ar[r]^(.55){\textup{res}} \ar[d]^{\textup{res}}& H^1(K_v^{\textup{ur}},T_{k,m})\ar[d]^{\textup{id}}\\
H^1(K_\alpha(n)_{w^\prime},T_{k,m})\ar[r]^(.55){\textup{res}} & H^1(K_v^{\textup{ur}},T_{k,m})\\
H^1(K_\alpha(n)_{w^\prime},\TT)\ar[r]^(.55){\textup{res}}\ar[u] & H^1(K_v^{\textup{ur}},\TT)\ar[u]
}$$

Semi-local Shapiro's Lemma yields the upper square of the following commutative diagram:
$$\xymatrix{
{\displaystyle\oplus_{w|v}} H^1(K_v^{\textup{ur}}/(K_\alpha)_w,T_{k,m})\ar[r] \ar[d]^{\cong}& {\displaystyle\oplus_{w|v}} H^1((K_\alpha)_w,T_{k,m}) \ar[d]^{\cong}\\
H^1(K_v^{\textup{ur}}/K_v,\textup{Ind}_{K_\alpha/K}\,T_{k,m})\ar[r] \ar[d]& H^1(K_v,\textup{Ind}_{K_\alpha/K}\,T_{k,m})\ar[d]\\
H^1(K_v^{\textup{ur}}/K_v,T_{k,m,r})\ar[r] & H^1(K_v,T_{k,m,r})
}$$
(i) shows that $\{\textup{loc}_w(\kappa_{[n,\alpha]})\}_{w|v}$ is in the image of the uppermost horizontal arrow, which implies that $\textup{loc}_v(\kappa_{n})$ is in the image of the lowermost horizontal arrow. This completes the proof.
\end{proof}
\begin{prop}
\label{prop:localdividingn}
For $\lambda|n$, we have $\kappa_n \in H^1_{\textup{tr}}(K_\lambda,T_{k,m,r})$.
\end{prop}
\begin{proof}
This is standard, c.f., Lemma 1.7.3 and 2.3.4 of \cite{howard-heegner1}.
\end{proof}
\subsubsection{Local properties at $p$}
Let $v$ be a place of $K$ above $p$.
\begin{prop}
\label{prop:localpstep1}
For any place  $w$ of $K_\alpha$ above $v$, we have
%\begin{itemize}
%\item[(i)]
$$\textup{loc}_w(\kappa_{[n,\alpha]}) \in \ker\left(H^1((K_\alpha)_w,T_{k,m})\lra H^1((K_\alpha)_w,\textup{F}_v^-(T_{k,m}))\right.$$
%\item[(ii)] $\textup{loc}_v(\kappa_{n}) \in \ker\left(H^1(K_v,T_{k,m,r})\lra H^1(K_v,\textup{F}_v^-(T_{k,m,r}))\right)$.
%\end{itemize}
\end{prop}
\begin{proof}
%Let $w^\prime$ be any place of $K_\alpha(n)$ above $w$.
Let $a\geq \alpha$ be a positive integer. The fact that $\frak{z}_{n,a}\in H^1_{\FF_{\textup{Gr}}}(K_{a}(n),\TT)$ and the $G(n)$-equivariance of the map
$$H^1(K_{a}(n)_v,\TT):=\bigoplus_{\wp|v}H^1(K_{a}(n)_{\wp},\TT)\lra \bigoplus_{\wp|v}H^1(K_{a}(n)_{\wp},\textup{F}_v^-(\TT))=H^1(K_{a}(n)_v,\textup{F}_v^-(\TT))$$
shows that
$$\textup{loc}_v(\frak{z}_{n,a}^\prime) \in \ker\left(H^1(K_{a}(n)_v,\TT)\lra H^1(K_{a}(n)_v,\textup{F}_v^-(\TT)\right).$$
This, along with the commutative diagram
$$
\xymatrix{H^1(K_{a}(n)_v,\TT)\ar[r]\ar[d]&H^1(K_{a}(n)_v,,\textup{F}_v^-(\TT))\ar[d]\\
H^1(K_{a}(n)_v,T_{k,m})\ar[r]&H^1(K_{a}(n)_v,,\textup{F}_v^-(T_{k,m}))
}$$
proves that $\textup{loc}_v(\kappa_{[n,a]}^\prime) \in \ker\left(H^1(K_{a}(n)_v,T_{k,m})\ra H^1(K_{a}(n)_v,\textup{F}_v^-(T_{k,m})\right)$, and hence that
$$\textup{loc}_{w}(\kappa_{[n,a]}) \in \ker\left(H^1((K_{a})_w,T_{k,m}) \lra H^1(K_{a}(n)_{w^\prime},\textup{F}_v^-(T_{k,m})\right)$$
for every prime $w$ of $K_a$ above $v$ and $w^\prime$ of $K_a(n)$ above $w$.
Let $c_{[n,a]}$ be the image of $\textup{loc}_w(\kappa_{[n,a]})$ under
$$H^1((K_{a})_w,T_{k,m}) \lra H^1((K_{a})_{w},\textup{F}_v^-(T_{k,m})).$$
We wish to prove that $c_{[n,\alpha]}=0$.
The map
$$H^1((K_{a})_w,T_{k,m}) \lra H^1(K_{a}(n)_{w^\prime},\textup{F}_v^-(T_{k,m}))$$
factors as
$$\xymatrix{H^1((K_{a})_w,T_{k,m}) \ar[rr]\ar[rd]&&H^1(K_{a}(n)_{w^\prime},\textup{F}_v^-(T_{k,m}))\\
&H^1((K_{a})_{w},\textup{F}_v^-(T_{k,m}))\ar[ru]_{\rho_a}&
}$$
which in return shows that $c_{[n,a]} \in \ker(\rho_a)$. Using the norm compatibility of the classes $\kappa_{[n,a]}$ (and hence that of $c_{[n,a]}$) as $a$ varies, it follows that $c_{[n,\alpha]}=\textup{cor}_{_{K_{a}(n)_{w^\prime}/(K_{a})_{w}}}(c_{[n,a]})$ and thus it suffices to show that $ \displaystyle \varprojlim_{a}\, \ker(\rho_a)=0$ where the inverse limit is with respect to the corestriction maps. By inflation-restriction
$$\ker(\rho_a) \cong H^1\left(K_{a}(n)_{w^\prime}/(K_{a})_{w}, H^0(K_{a}(n)_{w^\prime}, \textup{F}_v^-(T_{k,m}))\right),$$
and we are therefore reduced to checking the vanishing
$$\varprojlim_a H^0(K_{a}(n)_{w^\prime}, \textup{F}_v^-(T_{k,m}))=0.$$
But this is clear, as the size of the modules $H^0(K_{a}(n)_{w^\prime}, \textup{F}_v^-(T_{k,m}))$ are bounded independently of $a$, hence these modules stabilize for large enough $a$ and the corestriction maps then are multiplication by powers of $p$.
\end{proof}
\begin{rem}
\label{rem:equivgreenberg}
Proposition~\ref{prop:localpstep1} is equivalent to saying that
$$\textup{loc}_w(\kappa_{[n,\alpha]}) \in \textup{im}\left(H^1((K_\alpha)_w,,\textup{F}_v^+(T_{k,m}))\lra H^1((K_\alpha)_w,T_{k,m})\right).$$
\end{rem}
\begin{cor}
\label{cor:localpstep1}
$\textup{loc}_v(\kappa_{n}) \in \textup{im}\left(H^1(K_v,\textup{F}_v^+(T_{k,m,r}))\lra H^1(K_v,T_{k,m,r})\right).$
\end{cor}
\begin{proof}
This follows at once from Remark~\ref{rem:equivgreenberg} and the following commutative diagram which we obtain using Shapiro's Lemma:
$$\xymatrix{H^1((K_\alpha)_w,\textup{F}_v^+(T_{k,m}))\ar[d]\ar[r]& H^1((K_\alpha)_w,T_{k,m})\ar[d]\\
H^1(K_v,\textup{F}_v^+(T_{k,m,r}))\ar[r]& H^1(K_v,T_{k,m,r})
}$$
\end{proof}

Note that Corollary \ref{cor:localpstep1} alone is not enough to conclude that
$$\textup{loc}_v(\kappa_{n}) \in H^1_{\FF_{\textup{Gr}}}(K_v,T_{k,m,r}):=\textup{im}\left(H^1(K_v,\textup{F}_v^+(\TT\otimes\LL^{\textup{ac}}))\lra H^1(K_v,T_{k,m,r}) \right).$$
Consider the following hypothesis, which may be thought of as a condition to avoid \emph{trivial zeros} at characters of $\Gamma^{\textup{ac}}$ of finite order:

(\textbf{H.stz}) $H^0(K_v,\textup{F}_v^-(\overline{T}))=0$.

We assume until the end that \textbf{H.stz} holds. It follows by local duality and the fact that the cohomological dimension of $G_v$ is 2 (and using Nakayama's Lemma) that
$$H^2(K_v,\textup{F}_v^+(\TT\otimes\LL^{\textup{ac}}))=0,$$ and hence we have a surjection
$$H^1(K_v,\textup{F}_v^+(\TT\otimes\LL^{\textup{ac}})\twoheadrightarrow H^1(K_v,\textup{F}_v^+(T_{k,m,r})).$$
This shows that
$$\textup{im}\left(H^1(K_v,\textup{F}_v^+(\TT\otimes\LL^{\textup{ac}}))\ra H^1(K_v,T_{k,m,r})\right)=\textup{im}\left(H^1(K_v,\textup{F}_v^+(T_{k,m,r}))\ra H^1(K_v,T_{k,m,r})\right)$$
and thus
\begin{cor}
\label{cor:greenbergatphstz}
If one assumes \textbf{\upshape H.stz} then $\textup{loc}_v(\kappa_{n}) \in H^1_{\FF_{\textup{Gr}}}(K_v,T_{k,m,r})$.
\end{cor}

\begin{rem}
\label{rem:hstzharsh?}
The reader familiar with \cite{fouquetRIMS} will notice that the arguments that go into the proofs of Proposition~\ref{prop:localpstep1} and Corollary~\ref{cor:greenbergatphstz} are similar to those of \cite[\S5]{fouquetRIMS}, and that the hypothesis \textbf{\upshape H.stz} is not assumed in loc.cit. The point of this remark is to explain why \textbf{\upshape H.stz} is indeed necessary also in the proof of \cite[Lemma 5.14]{fouquetRIMS} (in fact implied by the arguments therein), more precisely to indicate how the norm-compatibility of $\{\kappa_{[n,a]}\}$ as $a$ varies cannot be used alone in order to conclude with Corollary~\ref{cor:greenbergatphstz} without assuming \textbf{\upshape H.stz}.

For every positive integer $a>>0$, fix a place $w_a$ of $K_a$ above $v$, such that $w_{a^\prime}\mid w_a$ for $a\geq a^\prime$. Let $\phi_a$ be the natural map
$$\phi_a:\,H^1((K_a)_{w_a},\textup{F}_v^+(T_{k,m})) \lra H^1((K_a)_{w_a},T_{k,m}).$$

Set $T_k=\TT/x^k\TT$ and define $r_a,\delta_a$
$$H^1((K_{a})_{w_{a}},\textup{F}_v^+(T_{k}))\stackrel{r_a}{\lra}H^1((K_a)_{w_a},\textup{F}_v^+(T_{k,m}))\stackrel{\delta_a}{\lra} H^2((K_a)_{w_a},\textup{F}_v^+(T_{k}))$$
as  the natural homomorphisms in the $G_{(K_a)_{w_a}}$-cohomology of the short exact sequence
$$0\lra \textup{F}_v^+(T_k)\stackrel{\pi^m}{\lra}\textup{F}_v^+(T_k)\lra \textup{F}_v^+(T_{k,m})\lra 0.$$
By Remark~\ref{rem:equivgreenberg}, there exists $\frak{x}_a\in H^1((K_a)_{w_a},\textup{F}_v^+(T_{k,m}))$ such that
$\phi_a(\frak{x}_a)=\textup{loc}_{w_a}(\kappa_{[n,a]})$. Note that $\phi_a$ is injective only if \textbf{H.stz} holds true and hence $\frak{x}_a$ is not uniquely determined. In particular, the collection $\{\frak{x}_a\}$ is not necessarily norm-coherent as $a$ varies. This is first of the problems. In order to check that
\begin{align*}
\textup{loc}_{w_a}(\kappa_{[n,a]}) \in H^1_{\FF_{\textup{Gr}}}((K_a)_{w_a},T_{k,m})&=\textup{im}\left(H^1((K_{a})_{w_{a}},\textup{F}_v^+(\TT))\ra H^1((K_a)_{w_a},T_{k,m})\right)\\
&\subset \textup{im}\left(H^1((K_{a})_{w_{a}},\textup{F}_v^+(T_{k}))\ra H^1((K_a)_{w_a},T_{k,m})\right)
\end{align*}
one attempts to choose $\frak{x}_a$ in a way that $\delta_a(\frak{x}_a)=0$ as follows: Argue (using the fact that the module  $H^2((K_a)_{w_a},\textup{F}_v^+(T_{k}))$ is of finite order bounded independently of $a$, when $x$ is \emph{unexceptional} in an appropriate sense) that there is a $\frak{b}>>0$ such that
\be\label{eqn:h2toh2vanish}
\textup{im}\left(H^2((K_{\frak{b}})_{w_{\frak{b}}},\textup{F}_v^+(T_{k}))\stackrel{\textup{cor}}{\lra} H^2((K_{a})_{w_{a}},\textup{F}_v^+(T_{k}))\right)=0,
\ee
so that for $\frak{x}_a$ chosen as $\frak{x}_a=\textup{cor}(\frak{x}_{\frak{b}})$ for a $\frak{b}$ satisfying (\ref{eqn:h2toh2vanish}), we would conclude that $$\delta_a(\frak{x}_a):=\textup{cor}(\delta_{\frak{b}}(\frak{x}_{\frak{b}}))=0.$$
However, if choosing $\frak{b}$ as in (\ref{eqn:h2toh2vanish}) was possible, then it would follow that \be\label{eqn:goeswrong}H^2((K_{a})_{w_{a}},\textup{F}_v^+(T_{k}))=0,\ee
since we have a commutative diagram
$$\xymatrix{H^2((K_{a})_{w_{a}},\textup{F}_v^+(T_{k})\otimes\LL^{\textup{ac}})\ar[rd]\ar@{->>}[rr]&&H^2((K_{a})_{w_{a}},\textup{F}_v^+(T_{k}))\\
&H^2((K_{\frak{b}})_{w_{\frak{b}}},\textup{F}_v^+(T_{k}))\ar[ur]_{\textup{cor}}&}$$
where the surjection is because the cohomological dimension of $G_{(K_{a})_{w_{a}}}$ is 2. Now by local duality, it is easy to see that (\ref{eqn:goeswrong}) is equivalent to asking \textbf{H.stz}.
\end{rem}

\subsubsection{Local properties at primes dividing $N$}
Throughout this section Assumption \ref{assume:specialtam} is in effect; see Remark~\ref{rem:tamtrivial} for the content of this assumption. Suppose $n\in \NN_{k,m,r}$. Let $v\mid N$ be a place of $K$ and $w$ be any place of $K_\alpha$ above $v$.
\begin{prop}$\,$
\label{prop:localatN}
\begin{itemize}
\item[(i)] $\textup{loc}_w(\kappa_{[n,\alpha]}) \in \ker(H^1((K_\alpha)_w,T_{k,m})\lra H^1((K_\alpha)_w^{\textup{ur}},T_{k,m}))$,
\item[(ii)] $\textup{loc}_v(\kappa_{n}) \in \ker(H^1(K_v,T_{k,m,r})\lra H^1(K_v^{\textup{ur}},T_{k,m,r}))$.
\end{itemize}
\end{prop}
\begin{proof}
The proof of Proposition~\ref{prop:localgoodplaces} goes through verbatim, except in the final paragraph one needs to replace $T_{k,m}$ (resp., $T_{k,m,r}$) by $T_{k,m}^{I_v}$ (resp., $T_{k,m,r}^{I_v}$), when they appear in the cohomology computed for the group $\textup{Gal}(K_v^{\textup{ur}}/K_v)$.
\end{proof}
\begin{prop}
\label{prop:localatN2}
$\textup{loc}_v(\kappa_n) \in H^1_{\FF_{\textup{Gr}}}(K_v,T_{k,m,r})$. (That is to say, $\textup{loc}_v(\kappa_n)$ is in the image of $H^1_{\textup{ur}}(K_v,\TT\otimes\LL^{\textup{ac}})$ under the natural map induced from $\TT\otimes\LL^{\textup{ac}} \twoheadrightarrow T_{k,m,r}$.
\end{prop}
\begin{proof}
The commutative diagram with exact rows
$$\xymatrix{H^1_{\textup{ur}}(K_v,\TT\otimes\LL^{\textup{ac}})\ar[r]& H^1(K_v,\TT\otimes\LL^{\textup{ac}}) \ar[r]\ar[d]&H^1(K_v^{\textup{ur}},\TT\otimes\LL^{\textup{ac}})\ar[d]\\
H^1_{\textup{ur}}(K_v,T_{k,m,n})\ar[r]&H^1(K_v,T_{k,m,n}) \ar[r]&H^1(K_v^{\textup{ur}},T_{k,m,n})
}$$
shows that the vertical arrow on the left induces a map
\be\label{eqn:urtourmap}H^1_{\textup{ur}}(K_v,\TT\otimes\LL^{\textup{ac}})\lra H^1_{\textup{ur}}(K_v,T_{k,m,r}).\ee
To conclude the proof, it suffices to prove that this map is surjective. On the other hand, the map (\ref{eqn:urtourmap}) is
$$H^1(K_v^{\textup{ur}}/K_v, (\TT\otimes\LL^{\textup{ac}})^{I_v}) \lra H^1(K_v^{\textup{ur}}/K_v, T_{k,m,r}^{I_v}),$$
and this map is surjective thanks to Lemma~\ref{lemma:surjective} below and the facts that
\begin{itemize}
\item $I_v$ acts trivially on $\LL^{\textup{ac}}$,
\item the cohomological dimension of $\textup{Gal}(K_v^{\textup{ur}}/K_v)$ is one.
\end{itemize}
\end{proof}
\begin{lemma}
\label{lemma:surjective}
The natural map $\TT^{I_v}\ra T_{k,m}^{I_v}$ is surjective.
\end{lemma}
\begin{proof}
Under our running assumptions, $\TT$ fits in an exact sequence of $\RR[[G_v]]$-modules
$$0\lra \RR(1)\otimes\mu \lra \TT \lra \RR\lra 0,$$
where $\mu^2=1$ and $\mu$ is unramified.

In the case $\mu\neq\textup{id}$, the proof of Proposition~\ref{prop:h1vanishes} shows that
$$\TT=(\RR(1)\otimes\mu)\oplus\RR$$
as $G_v$-modules and thus $\TT^{I_v}=\TT$ and $\TT^{I_v}=\TT \twoheadrightarrow T_{k,m}$ as desired.

In the case $\mu=\textup{id}$, the $I_v$-cohomology of the short exact sequence
$$0\lra R(1)\lra T \lra R \lra 0$$
(for $R=\RR,R_{k,m}$ and $T=\TT,T_{k,m}$) yields the following commutative diagram with exact rows:
$$
\xymatrix{0\ar[r]&\RR(1)\ar@{->>}[d]\ar[r]^{\alpha_0}&\TT^{I_v}\ar[r]^{\beta}\ar[d]&\RR\ar[d]\ar[r]^(.33){\partial}&H^1(I_v,\RR(1))\ar[d]\ar[r]^(.56){\alpha}&H^1(I_v,\TT)\ar[d]\\
0\ar[r]&R_{k,m}(1)\ar[r]_{\bar{\alpha}_0}&T_{k,m}^{I_v}\ar[r]_{\bar{\beta}}&R_{k,m}\ar[r]_(.33){\bar{\partial}}&H^1(I_v,R_{k,m}(1))\ar[r]_(.56){\bar{\alpha}}&H^1(I_v,T_{k,m})}$$
Under the running assumptions, the proof of Lemma~\ref{lemma:torsionfree} shows that the map $\partial$ is surjective, in particular non-zero and therefore injective. This shows the map $\beta$ is the zero map, hence $\alpha_0$ is surjective and therefore an isomorphism. The proof of Proposition~\ref{prop:tamuniformity} shows that $\bar{\alpha}$ is injective and it follows as above that the map $\bar{\alpha}_0$ is an isomorphism. Proof now follows by the surjectivity of the left-most vertical arrow.
\end{proof}

\begin{thm}
\label{thm:KS}
There is a Kolyvagin system $\pmb{\tilde{\kappa}} \in \overline{\textup{\textbf{KS}}}(\TT\otimes\LL^{\textup{ac}},\FF_{\textup{Gr}})$ such that 
$$\tilde{\kappa}_1=\kappa_1=\{\frak{z}_s\} \in \varprojlim_s H^1(K_s,\TT).$$
\end{thm}

\begin{proof}
Recall that 
$$\overline{\textup{\textbf{KS}}}(\TT\otimes\LL^{\textup{ac}},\FF_{\textup{Gr}})=\varprojlim {\textup{\textbf{KS}}}(T_{k,m,r},\FF_{\textup{Gr}},\PP_{k,m,r}).$$
Denote by $\kappa_n^{(k,m,r)}$ what we have called $\kappa_n$  in Definition~\ref{def:KSkmr2}. We have verified above that $\kappa_n^{(k,m,r)} \in H^1_{\FF_\textup{Gr}(n)}(K,T_{k,m,r})$. To finish off the proof one compares, as carried out in \cite[\S7]{nek92}, the images of the classes $\textup{loc}_{\ell}\left(\kappa_n^{(k,m,r)}\right)$ and $\textup{loc}_{\ell}\left(\kappa_{n\ell}^{(k,m,r)}\right)$ under the identifications
\be
\label{eqn:fsexplicit}
H^1_f(K_\lambda,T_{k,m,r}) \stackrel{\sim}{\lra} T_{k,m,r} \stackrel{\sim}{\longleftarrow} H^1_{s}(K_\lambda,T_{k,m,r}),
\ee
  and modifies $\kappa_{n}^{(k,m,r)}$ slightly as in \cite[Theorem 1.7.5]{howard-heegner1} so as to obtain 
  $$\tilde{\kappa}_{n}^{(k,m,r)} \in H^1_{\FF_{\textup{Gr}}(n)}(K,T_{k,m,r})\otimes \mathcal{G}(n)$$
   satisfying the desired condition
$$\phi_\lambda^{\textup{fs}}\left(\left(\textup{loc}_\ell(\tilde{\kappa}_{n}^{(k,m,r)}\right)\right)=\textup{loc}_\ell\left(\tilde{\kappa}_{n\ell}^{(k,m,r)}\right).$$ 

%See Olivier's EulerSystemforTower paper, page  
\end{proof}
\section{Howard's main conjecture}
Let $R_\infty$ be the ring $\RR\otimes_{\ZZ_p}\LL^{\textup{ac}}$. In this section, we record the standard application of the big Heegner point Kolyvagin system $\pmb{\tilde{\kappa}}$ that we have constructed in \S\ref{sec:EStoKS}. We omit the proofs as they follow closely the proofs in \cite[\S6.3]{fouquetRIMS}, except that we do not have the unwanted factor $\alpha \in \RR$ that appear in the statements\footnote{The attentive reader will notice that the extra factor $\alpha$ is not explicit in the statement of \cite[Theorem B(iii)]{fouquetRIMS}, however that Fouquet's element $z_\infty$ differ from the $\frak{z}_{\infty}$ defined below by a factor of $\alpha$.}${}^{,}$\footnote{Note that the extra factor $\alpha$ in Fouquet's\cite{fouquetRIMS} arguments is needed to obtain Kolyvagin systems for each specialization of $\TT$. Once one obtains those (in our case, they descend from our big Heegner point Kolyvagin system), the arguments of \S6 in loc.cit. carry out verbatim.} of \cite[Theorem B(iii), Theorem 3]{fouquetRIMS}.

For a finite extension $L$ of $K$, let $\tilde{H}^i_{f}(L, \TT)$ be Nekov\'a\v{r}'s extended Selmer group defined in \cite[\S6]{nek} and let
$$\tilde{H}^i_{f,\textup{Iw}}(K_\infty, \TT)=\varprojlim_s \tilde{H}^i_{f}(K_s, \TT).$$
It follows by \cite[Lemma 9.6.3]{nek} and \cite[Lemma 2.4.4]{howard} that $\tilde{H}^i_{f}(K_s, \TT)=H^1_{\FF_\textup{Gr}}(K_s,\TT)$ and hence we may define an element
  $$\frak{z}_\infty=\{\frak{z}_s\} \in \tilde{H}^1_{f,\textup{Iw}}(K_\infty, \TT).$$

Assume henceforth that the ring $R_\infty$ is a regular ring, so that $(R_\infty)_{\frak{p}}$ is a DVR for every height one prime $\frak{p}$ of $R_\infty$. If $M$ is a finitely generated torsion $R_\infty$-module, we may then define the characteristic ideal
$$\textup{char}(M)=\prod_{\frak{p}}\frak{p}^{\textup{length}(M_\frak{p})}.$$

For a general $R_\infty$-module $M$, let $M_{\textup{tors}}$ denote the $R_\infty$-torsion submodule.

\begin{thm}
\label{thm:mainappl}
$$\textup{char}\left(\tilde{H}^2_{f,\textup{Iw}}(K_\infty, \TT)_{\textup{tors}}\right)\,\mid \, \textup{char}\left(\tilde{H}^1_{f,\textup{Iw}}(K_\infty, \TT)/R_\infty\frak{z}_\infty\right)^2.$$
\end{thm}

It is reasonable to expect that one could adapt the arguments of \cite{arnold} in order to prove a similar statement assuming only that the ring $\RR$ is normal.

\subsection*{Acknowledgements.}
\label{sec:applications} The author wishes to thank Olivier Fouquet, Barry Mazur  and Jan Nekov\'{a}\v{r}, discussions with whom led the author to this work. He also thanks the anonymous referee for his comments and suggestions.

The author acknowledges a Marie Curie
Grant EU-FP7 230668 as well as partial support from T\"UBA, T\"UB\.ITAK and FONDECYT in the duration of this project.
%\newpage
%\appendix
%\section{Linear Algebra}
%\label{appendix:linearalgebra}
%%%%%%%%%%%%%%%%%%%%%%%%%%%%%%%%%%%%%%%%%%%%%%%%
{\scriptsize
\bibliographystyle{halpha}
\bibliography{references}
}
\end{document}